\documentclass[12pt, reqno]{amsart}
\usepackage{fullpage,url,amssymb,enumerate,colonequals, numprint}
\usepackage{hyperref}
\usepackage{comment}

\newcommand{\lmfdbec}[3]{\href{http://www.lmfdb.org/EllipticCurve/Q/#1#2#3}{{\text{\rm#1#2#3}}}}
\newcommand{\lmfdbeciso}[2]{\href{http://www.lmfdb.org/EllipticCurve/Q/#1/#2}{\text{\rm#1#2}}}

\newcommand{\Sage}{{\sc SageMath}}
\newcommand{\Magma}{{\sc Magma}}

\newcommand{\condS}{condition~{\bf (S)}}

\newcommand{\F}{\mathbb{F}}
\newcommand{\Fp}{\mathbb{F}_p}
\newcommand{\Fpstar}{\mathbb{F}_p^*}
\newcommand{\Fbar}{{\overline{\F}}}
\newcommand{\PP}{\mathbb{P}}
\newcommand{\Q}{\mathbb{Q}}
\newcommand{\Z}{\mathbb{Z}}
\newcommand{\Kbar}{{\overline{K}}}
\newcommand{\rhobar}{{\overline{\rho}}}
\newcommand{\frp}{{\mathfrak p}}
\newcommand{\eps}{\varepsilon}
\newcommand{\calL}{\mathcal{L}}
\newcommand{\calM}{\mathcal{M}}
\newcommand{\calO}{\mathcal{O}}
\DeclareMathOperator{\Aut}{Aut}
\DeclareMathOperator{\End}{End}
\DeclareMathOperator{\Gal}{Gal}
\DeclareMathOperator{\sss}{ss}
\DeclareMathOperator{\Tr}{Tr}
\newcommand{\vv}{\upsilon}
\newcommand{\GL}{\operatorname{GL}}
\newcommand{\PGL}{\operatorname{PGL}}
\newcommand{\PSL}{\operatorname{PSL}}
\newcommand{\SL}{\operatorname{SL}}
\DeclareMathOperator{\id}{id}
\def\legendre#1#2{\left(\displaystyle\frac{#1}{#2}\right)}
\newcommand{\diag}{{\operatorname{diag}}}
\newcommand{\mat}[4]{{\left(\begin{smallmatrix} #1 & #2 \\ #3 & #4 \end{smallmatrix} \right)}}

\numberwithin{equation}{section}
\newtheorem{theorem}{Theorem}[section]
\newtheorem{lemma}[theorem]{Lemma}
\newtheorem{corollary}[theorem]{Corollary}
\newtheorem{proposition}[theorem]{Proposition}
\newtheorem{example}[theorem]{Example}
\theoremstyle{remark}
\newtheorem{remark}[theorem]{Remark}

\begin{document}

\title{Global methods for the symplectic type of congruences between elliptic curves} 

\author{John Cremona}
\address{Mathematics Institute,
         University of Warwick,
         Coventry CV4 7AL,
         United Kingdom}
\email{j.e.cremona@warwick.ac.uk}

\author{Nuno Freitas}
\address{Departament de Matem\`atiques i Inform\`atica,
Universitat de Barcelona (UB),
Gran Via de les Corts Catalanes 585,
08007 Barcelona, Spain}
\email{nunobfreitas@gmail.com}

\date{\today}

\keywords{Elliptic curves, Weil pairing, Galois representations, symplectic isomorphisms}

\thanks{JEC was supported by EPSRC Programme Grant EP/K034383/1
  \textit{LMF: L-Functions and Modular Forms}, and the Horizon 2020
  European Research Infrastructures project \textit{OpenDreamKit}
  (\#676541)}
\thanks{NF was supported by the European Union's
  Horizon 2020 research and innovation programme under the Marie
  Sk\l{l}odowska-Curie grant agreement No.\ 747808}

\begin{abstract}
We describe a systematic investigation into the existence of
congruences between the mod~$p$ torsion modules of elliptic curves
defined over $\Q$, including methods to determine the symplectic type
of such congruences. We classify the existence and symplectic type of mod~$p$ congruences between twisted elliptic curves over number fields, giving global symplectic criteria that apply in situations where the available local methods may fail.

We report on the results of applying our methods for
all primes~$p\ge7$
to the elliptic curves in the LMFDB database, which currently includes
all elliptic curves of conductor less than~$\numprint{500000}$.  We also show that while such congruences exist
for each $p\le17$, there are none for~$p \geq 19$ in the database, in line
with a strong form of the Frey-Mazur conjecture.
\end{abstract}

\maketitle

\section{Introduction}

Let $p$ be a prime, $K$ be a number field and $G_K = \Gal(\Kbar/K)$
the absolute Galois group of $K$.  Let $E$ and $E'$ be elliptic curves
defined over $K$, and write $E[p]$ and $E'[p]$ for their $p$-torsion
$G_K$-modules.

Let $\phi : E[p] \to E'[p]$ be an isomorphism of $G_K$-modules.  There
is an element $d(\phi) \in \Fpstar$ such that the Weil
pairings~$e_{E,p}$ and~$e_{E',p}$ satisfy
\[
e_{E',p}(\phi(P), \phi(Q)) = e_{E,p}(P, Q)^{d(\phi)}
\]
for all $P, Q \in E[p]$.  We say that $\phi$ is a {\em symplectic
  isomorphism} or an {\em anti-symplectic isomorphism} if $d(\phi)$ is
a square or a non-square modulo~$p$, respectively.  When two elliptic
curves have isomorphic $p$-torsion modules we say that there is a
mod~$p$ {\em congruence} between them, or that they are {\em
  congruent} mod~$p$ or $p$-{\em congruent}.

For example, suppose that~$\phi$ is induced by an isogeny (also
denoted~$\phi$) from $E$ to $E'$, of degree $\deg(\phi)$ coprime
to~$p$. Then, using standard properties of the Weil pairing,
we have
\[
  e_{E',p}(\phi(P), \phi(Q)) = e_{E,p}(P, \hat\phi\phi(Q)) =
  e_{E,p}(P, \deg(\phi)(Q)) = e_{E,p}(P, Q)^{\deg(\phi)},
  \]
  where $\hat{\phi}$ denotes the dual isogeny; 
  thus $d(\phi)=\deg\phi\pmod{p}$.
  Hence $\phi$ is symplectic or antisymplectic according as $\deg\phi$
is a quadratic residue or nonresidue mod~$p$, respectively.  We will refer to
this condition as the \emph{isogeny criterion}.

Given $G_K$-isomorphic modules $E[p]$ and~$E'[p]$ as above, it is
possible they admit isomorphisms with both symplectic types.  This
occurs if and only if $E[p]$ admits an anti-symplectic automorphism.
The following proposition, which follows from results
in~\cite{FKSym}, gives several equivalent conditions for this property.

\begin{proposition} \label{P:conditionS}
Let $E$ be an elliptic curve over a number field~$K$. Let~$p$ be an
odd prime, and $\rhobar_{E,p} : G_K \to \GL_2(\Fp)$ the representation
arising from the action of~$G_K$ on $E[p]$.
Let~$G=\rhobar_{E,p}(G_K)\subset \GL_2(\Fp)$ be the image
of~$\rhobar_{E,p}$.  Then the following are equivalent:
\begin{enumerate}
\item $E[p]$ does not admit anti-symplectic automorphisms;
\item $G$ is not contained in a (split or nonsplit) Cartan subgroup;
\item the centralizer of~$G$ in~$\GL_2(\Fp)$ contains only matrices
  with square determinant;
\item either (A) $G$ is non-abelian,

  or\ \ (B) $\rhobar_{E,p} \cong \left( \begin{smallmatrix} \chi & * \\ 0 & \chi \end{smallmatrix} \right)$ where $\chi : G_K \to \Fpstar$ is a character
 and $* \neq 0$.
\end{enumerate}
In particular, these conditions are satisfied when $\rhobar_{E,p}$ is
absolutely irreducible, since then condition~(A) holds.
\end{proposition}
\begin{proof} See~\cite[Lemma~6]{FKSym}, \cite[Lemmas~7~and~8]{FKSym},
  and the proof of~\cite[Corollary~3]{FKSym}.
 \end{proof}

We say that \emph{~$E[p]$ satisfies \condS} if and
only if one of the equivalent conditions of
Proposition~\ref{P:conditionS} is satisfied.  In the case $K=\Q$,
\condS\ is satisfied by $E[p]$ for all~$p\ge7$, by~\cite[Corollary~3
  and Proposition~2]{FKSym}.  Hence, when a $G_\Q$-isomorphism $\phi :
E[p] \simeq E'[p]$ exists, normally there is only one possible
symplectic type for any such~$\phi$.  For an example where both types
exist, take $p=5$ and $E$, $E'$ to be the curves\footnote{Throughout
  the paper we use Cremona labels for elliptic curves over~$\Q$; these
  curves may be found in the LMFDB (see \cite{lmfdb}).}
$\lmfdbec{11}{a}{1}$ and $\lmfdbec{1342}{c}{2}$, respectively
(see~\cite[Example~5.2]{FKSym} for more details).

It is then natural to consider triples $(E,E',p)$ where $E/K$ and
$E'/K$ are elliptic curves with isomorphic $p$-torsion such that the
$G_K$-modules isomorphisms $\phi : E[p] \rightarrow E'[p]$ are either
all symplectic or all anti-symplectic.  In this case, we will say that
the {\em symplectic type} of $(E,E',p)$ is respectively symplectic or
anti-symplectic.  The problem of determining the symplectic type of
$(E,E',p)$ over~$K=\Q$ was extensively studied by the second author
and Alain Kraus in~\cite{FKSym}.

The isogeny criterion gives an easy solution when $(E,E',p)$
arises from an isogeny $h \colon E \to E'$ of degree~$n$ coprime
to~$p$, since in such cases $d(h|_{E[p]}) = n$ and the symplectic type
of $(E,E',p)$ is symplectic if $n$ is a square mod~$p$ and
anti-symplectic otherwise.

Given a generic triple $(E, E', p)$, in principle, one
could compute the $p$-torsion fields of $E$ and $E'$,
write down the Galois action on $E[p]$ and $E'[p]$ and check if they
are symplectically or anti-symplectically isomorphic. However, the
degree of the $p$-torsion fields grows very fast with~$p$, making this
method not practical already over~$\Q$ for $p = 5$.

One way to circumvent this computational problem, at least over~$\Q$, is to use the methods
presented in~\cite{FKSym}. Indeed, the main objective of {\it loc. cit.} was to establish a complete list of {\em local symplectic criteria},
allowing one to determine the symplectic type of $(E,E',p)$ using only
standard information about the local curves $E/\Q_\ell$ and
$E'/\Q_\ell$ at a single prime $\ell \neq p$ and congruence conditions
on~$p$. Further, it is also proved in~\cite{FKSym}  that if the
symplectic type of~$(E,E',p)$ is encoded in local information at a
single prime $\ell \neq p$, then one of the local criteria will
successfully determine it.  
There are cases where the local methods
are insufficient: however, this can occur only when 
the representation~$\rhobar_{E,p} : G_\Q \to \GL_2(\Fp)$ attached to~$E$ 
has image without elements of order~$p$;
see \cite[Proposition~16]{FKSym} for an example.

This paper has the following main objectives. The first two concern
theoretical results and methods which apply to elliptic curves defined
over arbitrary number fields, while the last applies these methods to
the LMFDB database of elliptic curves over~$\Q$ (see \cite{lmfdb}):
\begin{itemize}
 \item[(i)] We give global methods to determine the symplectic type
   of~$(E,E',p)$ when the local methods of \cite{FKSym} may not apply.
 \item[(ii)] We study in detail the case of congruences between twists.
\item[(iii)] We systematically identify and determine the symplectic
  type of all congruences between the elliptic curves defined
  over~$\Q$ in the LMFDB for all~$p\ge7$.
\end{itemize}

Towards (i), we give in Section~\ref{S:statistics} a complete
resolution for the case~$p=7$ over~$\Q$ using modular curves; this is the most relevant case as explained in~\S\ref{S:motivation}.
For~(ii), in
Section~\ref{S:cong-twist} we give, for general~$p$, global criteria
for the existence of congruences, and to decide their symplectic type, when $E$
and~$E'$ are twists of each other; Theorem~\ref{T:twist} shows that
(under condition~{\bf(S)}) congruences between quadratic twists occur if
and only if the projective image is dihedral, and
Theorem~\ref{T:quadratic} establishes the symplectic type of such
congruences. Theorems~\ref{T:higherTwists} and~\ref{T:typeHigher}
study the existence and the symplectic type of congruences between
higher order twists.  As a special case of the latter results, we prove
the following.

\begin{theorem}
\label{T:higherTwistsQ}
Over~$\Q$ we have the following 
$p$-congruences between higher twists:
\begin{itemize}
\item Quartic twists between curves of the form $E_a:\ Y^2=X^3+aX$,
  with $a \in \Q^*$, which have $j$-invariant~$1728$:
  \begin{itemize}
    \item
   $E_a$ is symplectically $3$-congruent to $E_{-1/3a}$, and also
      anti-symplectically $3$-congruent to both $E_{-4a}$ and
      $E_{4/3a}$, the latter two curves being $2$-isogenous to the
      former;
    \item
      $E_a$ is symplectically $5$-congruent to $E_{5/a}$, and also
      anti-symplectically $3$-congruent to both $E_{-4a}$ and
      $E_{-20/a}$, the latter two curves being $2$-isogenous to the
      former.
  \end{itemize}
\item Sextic twists between curves of the form $E_b:\ Y^2=X^3+b$, with
  $b \in \Q^*$, which have $j$-invariant~$0$:
  \begin{itemize}
    \item
      $E_b$ is symplectically $5$-congruent to $E_{4/5b}$, and also
      anti-symplectically $5$-congruent to both $E_{-27b}$
      and~$E_{-108/5b}$, the latter two curves being $3$-isogenous to
      the former.
    \item
      $E_b$ is symplectically $7$-congruent to $E_{-28/b}$, and also
      anti-symplectically $7$-congruent to both $E_{-27b}$
      and~$E_{756/b}$, the latter two curves being $3$-isogenous to
      the former.
  \end{itemize}
\end{itemize}
\end{theorem}

For objective~(iii), we have implemented in {\Magma} \cite{magma} and
{\Sage} \cite{sage} the methods from objective~(i) together with those
from~\cite{FKSym}. We have used our code to classify the symplectic
types of all $p$-congruences for~$p\ge7$ between curves in the LMFDB,
namely all elliptic curves defined over~$\Q$ of conductor less
than~$\numprint{500000}$.  In Section~\ref{S:statistics} we give
details of these computations, including details of all the
congruences found in the database for $p$ in the range $7 \leq p \leq
17$.  We also include a discussion on how to determine whether $E[p]$
and~$E'[p]$ are isomorphic (ignoring the symplectic structure) in both
the irreducible and reducible cases. Finally, in
Section~\ref{S:Frey-Mazur}, we prove the following:

\begin{theorem}
 \label{T:cong19}
 Let $p > 17$ be a prime. Let $E/\Q$ and~$E'/\Q$ be elliptic curves
 with conductors at most~$\numprint{500000}$.  Suppose that $E[p]
 \simeq E'[p]$ as $G_\Q$-modules. Then $E$ and $E'$ are
 $\Q$-isogenous.
 \end{theorem}

The Frey-Mazur conjecture states there is a constant $C \geq 17$ such
that, if $E/\Q$ and $E'/\Q$ satisfy $E[p] \simeq E'[p]$ as
$G_\Q$-modules for some prime $p > C$, then $E$ and $E'$ are
$\Q$-isogenous. Theorem~\ref{T:cong19} shows that for curves of
conductor at most~$\numprint{500000}$, this holds with~$C=17$. In view
of this conjecture, any~$(E,E',p)$ with $p > C$ arises from an
isogeny, hence its symplectic type is easily determined by the isogeny
criterion.

\subsection{Modular parametrizations}
The modular curve $X(p)$ parametrizes elliptic curves with full
level~$p$ structure. It has genus~$0$ for $p=2,3,5$, genus~$3$ for
$p=7$ and genus~$\ge26$ for $p\ge11$.  Fixing an elliptic
curve~$E/K$, the curve $X_E(p)$, which is a twist of~$X(p)$ (and
hence has the same genus), parametrizes pairs $(E',\phi)$ such that
$\phi$ is a symplectic isomorphism~$E[p]\cong E'[p]$; similarly
$X_E^-(p)$ parametrizes antisymplectic isomorphisms. Note that the pair $(E,\id)$ constituted by $E$ itself and the identity map
gives a base point defined over~$K$ on~$X_E(p)$, while $X_E^-(p)$ may have no $K$-rational points.

It follows that congruences modulo~$p$ for $p\le 5$ are common.  There
are certainly many mod~$3$ and mod~$5$ congruences in the database,
but we have not searched for these systematically. Indeed, since
$X_E(3)$ and $X_E(5)$ have genus~0, for each fixed~$E$, there will
always be congruences that are not part of the database independently
of its range. In contrast, for primes~$p \geq 7$, the curve $X_E(p)$
has genus~$\geq 3$ and so for each~$E$ there are only finitely many
mod~$p$ congruences with~$E$, hence the database
might contain all such congruences; however proving this fact for a
fixed~$E$ is a hard problem.

For convenience 
we sometimes also write $X_E^+(p)$ to denote $X_E(p)$. 
By ``explicit equations'' for $X_E^{\pm}(p)$, we mean the
following.
\begin{itemize}
  \item An explicit model for a family of curves, with equations whose
    coefficients are polynomials in $\Q[a,b]$, such that specializing
    $a,b$ gives a model for $X_E^{\pm}(p)$ where $E$ is the elliptic
    curve with equation $Y^2=X^3=aX+b$; each rational point $P$ is
    either a cusp of the modular curve or encodes a pair $(E',\phi)$
    such that $(E,E', p)$ is a symplectic (respectively,
    antisymplectic) triple.
    \item A rational function with coefficients in $\Q(a,b)$ defining
      the map $j: X_E^{\pm}(p) \to \PP^1$, taking a point
      $P=(E',\phi)$ to $j(E')$.  The degree of this map is the
      index~$[\PSL_2(\Z):\Gamma(p)] = |\PSL_2(\Fp)|$: for example, when
        $p=7$ the degree is~$168$.
      \item Rational functions $c_4$ and $c_6$ such that for each
        non-cuspidal point~$P=(E',\phi)$, a model for $E'$ is
        $Y^2=X^3+a'X+b'$ where $(a',b')=(-27c_4(P),-54c_6(P))$.
\end{itemize}
For $p=3$ and~$p=5$, the curves themselves have genus~$0$ and hence we
do not need equations, but the formulas for $j$, $c_4$ and~$c_6$ are
still useful.  They may be found in \cite{Rubin-Silverberg},
\cite{Rubin-Silverberg2} and \cite{Silverberg} for the symplectic
case, and in \cite{FisherHessian} and \cite{FisherQuintic} for the
anti-symplectic case.

Equations for $X_E(7)$, in this sense, were obtained by Kraus and
Halberstadt in~\cite{Halberstadt-Kraus-XE7}, though the functions
$c_4$ and $c_6$ in \cite{Halberstadt-Kraus-XE7} are only defined away
from $5$ points (which may or may not be rational). Fisher gives more
complete equations for this, together with $X_E^-(7)$ and
$X_E^{\pm}(11)$, in~\cite{Fisher}.

\subsection{Further motivation} \label{S:motivation}
We finish this introduction with a discussion on how the methods of
this paper complement the local methods in~\cite{FKSym}.

From the discussion so far we know that triples $(E,E',p)$ as above
give rise to $\Q$-points on one of the modular curves~$X_E(p)$ or $X_E^-(p)$. 


As mentioned above, it follows from~\cite{FKSym} that if no local
symplectic criterion applies to~$(E,E',p)$ then the image of
$\rhobar_{E,p}$ is irreducible and contains no element of
order~$p$. Moreover, when $\rhobar_{E,p}$ is reducible
only the local criteria at primes of multiplicative or good reduction may
succeed and the bounds on a prime~$\ell$ for which a local
criterion at~$\ell$ applies may be very large.  From the
strong form of Serre's uniformity conjecture, these `bad' cases for
the local methods imply that either $E$ has complex multiplication
(CM), or one of the following holds:
\begin{itemize}
 \item[(i)] $p=3,5,7$ and $\rhobar_{E,p}$ is reducible or has image the normalizer of a Cartan subgroup;
 \item[(ii)] $p=11$ and $\rhobar_{E,p}$ has image the normalizer of non-split Cartan subgroup;
 \item[(iii)] $p=13$ and $\rhobar_{E,p}$ is reducible; 
 \item[(iv)] $p=13$ and $\rhobar_{E,p}$ has exceptional image projectively isomorphic to $S_4$;
 \item[(v)] $p=11,17$ or $37$, and $j(E)$ is listed
 in \cite[Table~2.1]{DahmenPhD}; in particular, $\rhobar_{E,p}$ is reducible.
\end{itemize}
In \cite[Corollary~1.9]{BarinderCrem} there is a list of~$3$ rational
$j$-invariants of elliptic curves over~$\Q$ satisfying~(iv); it has
recently been shown (see~\cite{BDMTV-S4}) that the associated
genus~$3$ modular curve~$X_{S_4}(13)$ found explicitly
in~\cite{BarinderCrem} has no more rational points, and hence that
this list is complete. By Theorem~\ref{T:twist}, none of these
curves is mod~$13$ congruent to a twist of another of them
(including itself); the same is true for the curves and values of~$p$
in case~(v). Thus no examples arise in cases (iv) and (v).

The case~(iii) includes the infinitely many curves with
$\rhobar_{E,13}$ reducible (recall that $X_0(13)$ has genus~0). However, 
we know of no reducible mod~$13$ congruences between rational elliptic curves, so there are no known examples in this case either. Nevertheless, in spite of the lack of helpful bounds as mentioned above, a putative congruence
between such curves will often be addressed by local methods in
practice. Indeed, the bounds are very large as they depend on
Tchebotarev density theorem to predict a prime~$\ell$ of good
reduction for~$E$ where Frobenius has order multiple of~$p$ but,
in practice, it is usually easy to find such a prime~$\ell$ after
trying a few small primes.
We refer to \cite[Example~31.2]{FKSym} for
an example with $p=7$ analogous to the discussion in this paragraph.

Our method described in Section~\ref{S:statistics} for $p=7$ could be
adapted by replacing the modular curves~$X_E(7)$ by $X_E(p)$ when
explicit equations for the latter are known, which is the case for
$p=3,5$ and~$11$. This method works independently of the image
of~$\rhobar_{E,p}$, so covers the remaining cases (i),~(ii) entirely and also case (v) with $p=11$.

\subsection{Notation} \label{S:notation}
For $p$ an odd prime, define $p^*=\pm p\equiv1\pmod4$, so that
$\Q(\sqrt{p^*})$ is the quadratic subfield of the cyclotomic
field~$\Q(\zeta_p)$.

Let $D_{n}$ denote the dihedral group with~$2n$ elements and $C_n
\subset D_{n}$ for a normal cyclic subgroup of order~$n$; note that
$C_n$ is unique unless $n=2$, in which case $D_{n} \simeq C_2 \times
C_2$ and there are three such subgroups.  We will also denote by $C
\subset \GL_2(\Fp)$ a Cartan subgroup (either split or non-split) and
by $N$ its normalizer in $\GL_2(\Fp)$.

For a number field~$K$ we denote by~$G_K$ the absolute Galois group
of~$K$.

Let $\eps_d$ be the quadratic character of $G_K$ associated to the
extension $K(\sqrt{d})/K$.

For $a, b \in K$, we write $E_{a,b}$ for the elliptic
curve defined over~$K$ by the short Weierstrass equation
$Y^2=X^3+aX+b$; every elliptic curve over~$K$ has such a model, unique
up to replacing $(a,b)$ by~$(au^4,bu^6)$ with $u\in K^*$.

We will denote by~$I$ the identity matrix in~$\GL_2(\Fp)$.

\subsection{Acknowledgments} We thank Alain Kraus and Samir Siksek for helpful conversations. We also thank Tom Fisher for pointing 
us to the work of Halberstadt~\cite{Halberstadt-11nonsplit}. We also thank the anonymous referee for the careful reading of the paper and various helpful comments.

\section{Congruences between Twists}\label{S:cong-twist}

Congruences between elliptic curves which are twists of each other
arise in a number of ways in our study; these are often between
quadratic twists but they also occur between higher order twists. In
this section we study all types of congruences between twists.

Let $p$ be an odd prime.

Let $K$ be a number field and $E/K$ an elliptic curve. If the
representation $\rhobar_{E,p}$ has image contained in the
normaliser~$N$ of a Cartan subgroup~$C$ of~$\GL_2(\Fp)$ but not
in~$C$ itself, then the projective image $\PP \rhobar_{E,p}(G_K)$ is
isomorphic to~$D_{n}$ for some~$n \geq 2$
(see~\cite[Theorem~XI.2.3]{LangModForms}). The preimage of~$C$
in~$G_K$ then cuts out a quadratic extension~$M = K(\sqrt{d})$
satisfying $\PP \rhobar_{E,p}(G_M) \simeq C_n$.  We will see in
Theorem~\ref{T:twist} that in this setting there is a mod~$p$
congruence between~$E$ and its quadratic twist by~$d$. This
construction is the main source of congruences between twists and it
appeared already in~\cite[\S2]{Halberstadt-11nonsplit}, including a
determination of the symplectic type of the congruence; our
contribution for this type of congruences is to show that congruences
between quadratic twists only occur via this construction, when the
projective image is dihedral: see Theorem~\ref{T:quadratic}. The
second contribution of this section is to classify congruences between
higher order twists in Theorem~\ref{T:higherTwists} and describe
their symplectic type in~Theorem~\ref{T:typeHigher}.

\subsection{Twists of elliptic curves}
\label{S:twists}
Here we recall some standard facts about elliptic curves and their twists.
Let $E$ be an elliptic curve defined over any field~$K$ of
characteristic~$0$.

The twists of $E$ over $K$ are parametrized by $H^1(G_K,\Aut(E))$.  If
$E'$ is a twist of~$E$, then by definition there exists a $\Kbar$-isomorphism
$t:E\to E'$ so that, for all $P\in E(\Kbar)$, we have 
$\sigma(t(P))=\psi(\sigma)t(\sigma(P))$ 
where~$\psi:G_K\to\Aut(E')\cong\Aut(E)$ is a cocycle.  Here,
$\Aut(E)\cong\mu_n$ is cyclic of order~$n=2$, $4$ or $6$ according as
$j(E)\not\in\{0,1728\}$, $j(E)=1728$ and $j(E)=0$ respectively.  By
Kummer theory, $H^1(G_K,\Aut(E))\cong H^1(G_K,\mu_n)\cong
K^*/(K^*)^n$; hence each $n$-twist is determined by a parameter~$u\in
K^*$ whose image in $K^*/(K^*)^n$ determines the isomorphism class of
the twist. 

In the cases where $\Aut(E)\not=\{\pm1\}$, we make two elementary, but
important, observations. First, $G_K$ acts non-trivially on~$\Aut(E)$,
unless $-1$ or $-3$ (respectively for the cases $n=4$ and $n=6$) are
squares in~$K$, so the cocycle is usually not a homomorphism;
secondly, there are \emph{two} isomorphisms
$\calO=\Z[\zeta_n]\cong\End(E)$ (differing by complex conjugation in
$\calO$), when $n=4$ or~$6$, and hence two actions of
$\calO^*\cong\Aut(E)$ on~$E$.  We fix isomorphisms $\calO\cong\End(E)$
and $\calO\cong\End(E')$, and hence isomorphisms
$\mu_n=\calO^*\cong\Aut(E)\cong\Aut(E')$, which are normalised (in the
sense of \cite[Prop.~I.1.1]{SilvermanII}), so that $\zeta\cdot
t(P)=t(\zeta\cdot P)$ for $P\in E(\Kbar)$ and $\zeta\in\calO^*$. Then
the twist isomorphism~$t$ is an isomorphism of~$\calO$-modules, and we
may view the twisting cocycle~$\psi$ as taking values in~$\calO^*$.

Explicitly, in terms of a short Weierstrass equation~$E_{a,b}$ for~$E$, we fix
the action of~$\zeta\in\calO^*$ to be
$(x,y)\mapsto(\zeta^2x,\zeta^3y)$, and the $n$-twist by~$u\in K^*$ to be
$t:(x,y)\mapsto(v^2x,v^3y)$ where $v \in \overline{K}$ satisfies $v^n=u$. Then the associated cocycle is
$\psi(\sigma)=\sigma(v)/v$, we have
\begin{equation}\label{E:twist}
  \sigma(t(P)) = \psi(\sigma)t(\sigma(P)) = t(\psi(\sigma)\sigma(P))
   \quad \text{ for all } \sigma \in G_K,
  \; P \in E(\Kbar)
\end{equation}
and~$E$ and its $n$-twist by~$u$ become isomorphic
over~$K(\root n\of u)$, which is an extension of~$K$ of degree
dividing~$n$.
We end this subsection with a brief discussion of each kind of twists.

\emph{Quadratic twists} ($n=2$): we denote by~$E^d$ the quadratic
twist of~$E$ by~$d\in K^*$.  If $E=E_{a,b}$ then $E^d=E_{ad^2,bd^3}$.

\emph{Higher twists} ($n=4$ and $n=6$): Quartic twists only exist when
$j(E)=1728$.  The short Weierstrass model of such a curve has the
form~$E_{a,0}$, and its quartic twist by~$u$ is $E_{au,0}$.  When
$u=d^2$, the quartic twist by~$u$ is the same as the quadratic twist
by~$d$.

Sextic twists only exist when $j(E)=0$.  The short Weierstrass model
of such a curve has the form~$E_{0,b}$, and its sextic twist by~$u$ is
$E_{0,bu}$.  The sextic twist by~$u=d^3$ is the same as the quadratic
twist by~$d$.

\begin{remark}\label{R:2-3-isog}
The quartic twist by~$u=-4$ is trivial if and only if $-1$ is a square
in~$K$, since $-4=(1+\sqrt{-1})^4$.  The curves $E_{a,0}$
and~$E_{-4a,0}$ are $2$-isogenous over~$K$ (see
\cite[p.~336]{SilvermanI}); over $K(\sqrt{-1})$ this $2$-isogeny is
the endomorphism $1+\sqrt{-1}$.  Note that this twist is \emph{not} a
quadratic twist, despite the fact that the curve and its twist become
isomorphic over a quadratic extension.  Taking $v=1+\sqrt{-1}$ so that
$v^4=-4$, the cocycle~$\psi$ takes value
$\psi(\sigma)=\sigma(1+\sqrt{-1})/(1+\sqrt{-1})=-\sqrt{-1}$ when
$\sigma$ does not fix~$\sqrt{-1}$.

The sextic twist by~$u=-27$ is the quadratic twist by~$-3$, and is
trivial if and only if $-3$ is a square in~$K$.  The curves $E_{0,b}$
and~$E_{0,-27b}$ are $3$-isogenous over~$K$; over $K(\sqrt{-3})$ this
$3$-isogeny is the endomorphism $\sqrt{-3}$.
\end{remark}

\subsection{The mod~$p$ Galois representations of twists}\label{SS:GalRepTwist}
We now consider the effect of twisting~$E$ on the associated mod~$p$ Galois
representations~$\rhobar_{E,p}$.
This is straightforward in the case of quadratic
twists, but more involved for higher twists.

Let~$K$ be a number field. Let $E$ and~$E'$ be elliptic curves
over~$K$ having the same $j$-invariant $j = j(E) = j(E')$. Assume they
are not isomorphic over~$K$. Then $E$ and $E'$ are (non-trivial)
twists and become isomorphic over an extension $L/K$. Write $d = [L :
  K]$.

From the discussion in section~\ref{S:twists}, we have a twist map $t : E \to E'$ with
an associated cocycle~$\psi : G_K \to \mu_m \subseteq\calO^*$, where
$m\in\{2,3,4,6\}$, the \emph{order} of the twist, is the order the
subgroup of~$\calO^*$ generated by~$\psi(G_K)$.  We have the following
cases:
\begin{itemize}
 \item[(i)] for arbitrary~$j$: $m=d=2$ (quadratic twists); and
   additionally,
 \item[(ii)] for $j = 1728$ only: $m=4$ and $d \in \{2,4\}$  (quartic
   twists); and
  \item[(iii)] for $j = 0$ only: $m=d\in\{3,6\}$ (cubic or sextic twists).
\end{itemize}

When $m=4$, the case $d=2$ occurs only for the special
quartic twist by~$-4$ as in Remark~\ref{R:2-3-isog}.

Denote by $t_p : E[p] \to E'[p]$ the restriction of~$t$ to the
$p$-torsion. Then~\eqref{E:twist} becomes
\begin{equation}\label{E:cocycle}
  \sigma(t_p(P)) = \psi(\sigma) \cdot t_p(\sigma(P))
                 = t_p(\psi(\sigma)\cdot\sigma(P)) 
   \quad \text{ for all } \sigma \in G_K,
  \; P \in E[p].
\end{equation}
Let $P_1, P_2$ be a basis of~$E[p]$, so that $t_p(P_1), t_p(P_2)$ is a
basis of~$E'[p]$.  With respect to these bases, the map~$t_p$ is
represented by the identity matrix in~$\GL_2(\Fp)$ and~$\psi(\sigma)$
by a matrix~$\Psi(\sigma)$ (the same matrix on both $E[p]$
and~$E'[p]$).  Then~\eqref{E:cocycle} implies, for all $\sigma \in
G_K$, the matrix equation
\begin{equation} \label{E:PsiMatrix}
 \rhobar_{E',p}(\sigma) = \rhobar_{E,p}(\sigma) \cdot \Psi(\sigma). 
\end{equation}
For quadratic twists, $\Psi(\sigma)=\pm I$, but in
general~$\Psi(\sigma)$ is not scalar.  Note that $\det\Psi(\sigma)=1$
in all cases, since the determinants of~$\rhobar_{E,p}$
and~$\rhobar_{E',p}$ are both given by the cyclotomic character.

The map $\Psi:G_K\to\SL_2(\Fp)$ becomes a homomorphism over an
extension $K'/K$ given by $K'=K$ if $m=2$, $K'=K(\sqrt{-1})$ if $m=4$
and $K'=K(\sqrt{-3})$ if $m=3,6$. In the quadratic case, $\Psi$
matches the quadratic character~$\varepsilon_d$ associated to the
quadratic extension~$K(\sqrt{d})$ of~$K$ over which the curves become
isomorphic.

In general, the representation attached to the twist of~$E$ is
obtained from that of~$E$ itself by twisting by the cocycle~$\Psi$, with
values in~$\GL_2(\Fp)$.  For quadratic twists this is just the tensor product
by a quadratic character.
We summarize this discussion in the following.

\begin{lemma}\label{L:twist-rep} \hfill
  \begin{enumerate}
\item Let $E'$ be the twist of~$E$ by a cocycle~$\psi$. Then there is
  an isomorphism $t :E\to E'$, defined over an extension of $K$ of
  degree at most~$6$, satisfying~\eqref{E:twist}.
  In the case of a quadratic twist with $E'=E^d$, the isomorphism $t$
  is defined over $K(\sqrt{d})$ and satisfies
\begin{equation}\label{E:iso-twist}
\sigma(t (P)) = \eps_d(\sigma)t (\sigma(P)) \qquad\text{for
  all~$\sigma\in G_K$ and $P\in E(\Kbar)$}.
\end{equation}
\item For all primes~$p$ we have in general the matrix
  equation~\eqref{E:PsiMatrix}.
In the case of quadratic twists this simplifies to
  \begin{equation} \label{E:rho-twist}
  \rhobar_{E^d,p} \cong \rhobar_{E,p}\otimes\eps_d.
\end{equation}
\end{enumerate}
\end{lemma}

Assuming that $E$ and $E'$ are $p$-congruent, there exists an
isomorphism $\phi : E[p] \to E'[p]$
of $G_K$-modules. 
Choosing
compatible bases for~$E[p]$ and~$E'[p]$ as above, let~$A \in
\GL_2(\Fp)$ be the matrix representing~$\phi$ with respect to them.
Then, using~\eqref{E:PsiMatrix}, we have
\begin{equation}\label{E:PsiM}
A \rhobar_{E,p}(\sigma) A^{-1} = \rhobar_{E',p}(\sigma) =
\rhobar_{E,p}(\sigma) \cdot \Psi(\sigma) \quad \text{ for all } \sigma
\in G_K
\end{equation}
and the 
symplectic type of~$\phi$ is determined by the square class of $\det A$; this latter conclusion follows 
from the fact that~$t_p$ preserves the Weil pairing
and~\cite[Lemma~6]{FKSym}. Moreover, if there is another $A' \in
\GL_2(\Fp)$  satisfying~\eqref{E:PsiM} then $\det A' = \det A \cdot
\lambda^2$ by Proposition~\ref{P:conditionS} (under the natural
assumption that~$E[p]$ satisfies condition~{\bf (S)}) and so the symplectic type of~$\phi$ is also determined by~$\det A'$ mod squares.

\subsection{Projectively dihedral images}
\label{SS:dihedral}
From Proposition~\ref{P:conditionS}, \condS\ follows from
absolute irreducibility.  Conversely, the next result shows that, in
the presence of a $p$-congruence between twists, \condS\
implies that $\rhobar_{E,p}$ is absolutely irreducible.  Clearly,
elliptic curves with CM can only satisfy \condS\ when the
CM is not defined over the base field~$K$, or when~$p$ ramifies in the
CM field, as otherwise the image is abelian.

\begin{lemma} \label{L:noCyclic}
Let $E/K$ be an elliptic curve and $p$ an odd prime such that~$E[p]$
satisfies \condS. Assume further that $j(E) \neq 0$ if $p=3$.  

If $E$ is $p$-congruent to a
twist~$E'$ then $\rhobar_{E,p}$ is absolutely irreducible.
\end{lemma}
\begin{proof}
For a contradiction, suppose $\rhobar_{E,p}$ is absolutely reducible. Thus
\begin{equation} \label{E:absRed}
  \rhobar_{E,p} \otimes \Fbar_p \cong 
 \begin{pmatrix}
 \chi_1 & h \\ 0 &\chi_2
 \end{pmatrix} \quad \text{ with } \; h \neq 0, 
\end{equation} 
where $h \neq 0$ follows directly from
Proposition~\ref{P:conditionS} part (4).


Note that we cannot have $j(E) =0$ or $1728$, as then
$E$ would have CM by~$\Q(i)$ or $\Q(\sqrt{-3})$, and the hypothesis
on~$p$ would imply that the image would be in the normalizer of a
Cartan subgroup, contradicting~\eqref{E:absRed}. Therefore $E$ only admits
quadratic twists, and we have $E' = E^d$ for some non-square $d \in
K^*$.

Let $\eps_d$ be the quadratic character associated to the extension
$K(\sqrt{d})/K$. From the hypothesis $E[p] \simeq E^d[p]$, part 2) of
Lemma~\ref{L:twist-rep} and~\eqref{E:absRed} it follows that $\chi_1
=\chi_1\eps_d$ since both give the Galois action on the unique fixed
line. This contradicts ~$\eps_d \neq 1$ as $p \neq 2$.
\end{proof}

Under the natural \condS, the previous lemma tells us we can assume
that $\rhobar_{E,p}$ is absolutely irreducible for our study of
congruences between twists. In fact more is true: the next proposition
says that congruences between twists arise when the projective image
is dihedral of order at least~4 (equivalently, when the image is
contained in the normaliser of a Cartan subgroup but not in the Cartan subgroup
itself), and only then.  The converse part of the following result is
already contained in~\cite[\S2]{Halberstadt-11nonsplit}, with a
different proof and without the uniqueness statement.

\begin{theorem}\label{T:twist}
Let $E/K$ be an elliptic curve and $p$ an odd prime such that~$E[p]$
satisfies \condS. 


1) Assume further $j(E) \neq 0$ if $p=3$. If $E$ is $p$-congruent to a twist, 
then the image of~$\rhobar_{E,p}$ is contained in the normaliser of a
Cartan subgroup of $\GL_2(\Fp)$ but not in the Cartan subgroup itself and
$\PP \rhobar_{E,p} \simeq D_n$ for some $n \geq 2$.

2) Conversely, if the image of~$\rhobar_{E,p}$ is absolutely irreducible and contained in the
normaliser of a Cartan subgroup of $\GL_2(\Fp)$ but not in the Cartan
subgroup itself, 
then we have $E[p]\cong E'[p]$ where $E'$ is the (non-trivial) twist associated to
the quadratic extension~$K(\sqrt{d})$ cut out by the Cartan subgroup.
This is the quadratic twist~$E^d$ unless $j(E)=1728$ and $d=-1$, in
which case it is the quartic twist by~$-4$.

Moreover, there is a unique such twist~$E'$ which is $p$-congruent to~$E$, except when the
projective image has order~$4$, in which case there are three (non-trivial) such
twists.
\end{theorem}
\begin{proof} 1) Let $E'$ be a $p$-congruent twist of~$E$. If $E'$ is
  a quartic or sextic twist of~$E$ then $j(E)=0,1728$ and $E$ has CM
  (not defined over~$K$, since the image is not in a Cartan subgroup by condition~{\bf (S)}) by
  $\Q(\sqrt{-1})$ or $\Q(\sqrt{-3})$, respectively; moreover, since $p$ is
  not ramified in the CM field, the image is inside the normalizer of
  a Cartan subgroup and is projectively
   isomorphic to~$D_n$ for some $n \geq 2$.

Therefore we can assume $E' = E^d$ is the quadratic twist of~$E$ by a
non-square $d\in K^*$.  From Lemma~\ref{L:noCyclic} we know
that~$\rhobar_{E,p}$ is absolutely irreducible.

Taking traces in~\eqref{E:rho-twist}, for all $\sigma\in G_K$ with
$\eps_d(\sigma) = -1$, the image~$\rhobar_{E,p}(\sigma)$ has
trace~$0$.  So the image $H=\rhobar_{E,p}(G_K)$ has a
subgroup~$H^+=\rhobar_{E,p}(G_{K(\sqrt{d})})$ of index~$2$ such that
all elements of~$H\setminus H^+$ have trace~$0$; such elements have
order~$2$ in $\PGL_2(\Fp)$. From the irreducibility of~$E[p]$ and the
classification of subgroups of $\PGL_2(\Fp)$ (\cite[Theorem
  XI.2.3]{LangModForms}), it follows that $H$ is contained in the
normaliser of a Cartan subgroup~$C$ and $H^+=H\cap C$.

2) For the converse, suppose that $\rhobar_{E,p}$ is absolutely irreducible and has image contained in
the normalizer $N \subset \GL_2(\Fp)$ of a Cartan subgroup~$C$, but
is not contained in~$C$ itself. Thus $\PP \rhobar_{E,p} \simeq D_n$ for $n \geq 2$ and
$\Tr \rhobar_{E,p}(\sigma) = 0$
for all $\sigma \in G_K$ such that $\rhobar_{E,p}(\sigma) \not\in C$.

Let $K(\sqrt{d})$ be the quadratic extension cut out
by~$\rhobar_{E,p}^{-1}(C)$, with associated character $\eps_d$ as
above.  First suppose that we are not in the special case where
$j(E)=1728$ and $K(\sqrt{d})=K(\sqrt{-1})$.  Set $E'=E^d$, the
quadratic twist.  By~(\ref{E:rho-twist}), for all $\sigma \in G_K$ we
have the following equality of traces
\[\Tr \rhobar_{{E^d},p}(\sigma) = \eps_d(\sigma) \cdot \Tr \rhobar_{E,p}(\sigma).\]

Clearly, if $\eps_d(\sigma)=1$ then $\Tr \rhobar_{E,p}(\sigma) = \Tr \rhobar_{E^d,p}(\sigma)$. 
If $\eps(\sigma) = -1$ then $\rhobar_{E,p}(\sigma) \in \rhobar_{E,p} (G_K) \backslash C$ and $\Tr \rhobar_{E,p}(\sigma) = 0$, so also $\Tr \rhobar_{E^d,p}(\sigma) = 0$.
Then, $\Tr \rhobar_{E,p}(\sigma) = \Tr \rhobar_{E^d,p}(\sigma)$ 
for all $\sigma \in G_K$. 
Since $\rhobar_{E,p}$ and $\rhobar_{E^d,p}$ are absolutely irreducible
and have the same traces, they are isomorphic.

In the special case, $E'$ is the quartic twist of~$E$ by~$-4$, since
these become isomorphic over~$K(\sqrt{-1})$; now $E$ and~$E'$ are
isogenous, so are $p$-congruent for all odd~$p$.

For the last part, we note that $D_2\cong C_2\times C_2$ has three
cyclic subgroups of index~$2$, while $D_n$ for $n\ge3$ has only one
such subgroup.
\end{proof}

\subsection{The symplectic type of congruences between quadratic twists}
\label{SS:typeQuadratic}
A special case of the situation described 
in Theorem~\ref{T:twist} is the case
of elliptic curves with~CM.  Here, the quadratic twists are isogenous
to the original curve so we may already determine the symplectic nature
of the congruence.  For simplicity we state such a result only over~$\Q$.

\begin{corollary}\label{C:CM-twist}
  Let $E/\Q$ be an elliptic curve with CM by the imaginary quadratic
  order of (negative) discriminant~$-D$.  Set $M=\Q(\sqrt{-D})$.

  \begin{enumerate}
  \item
    For $D\not=3,4$: $E[p] \simeq E^{-D}[p]$ for all primes $p \geq 5$
    unramified in~$M$.  This congruence
    is symplectic if and only if $\legendre{D}{p}=+1$.  For each
    such~$p$, $E^{-D}$ is the unique quadratic twist of~$E$
    which is $p$-congruent to~$E$.

  \item
    For $D=3$: $E[p] \simeq E^{-D}[p]$ for all primes $p
    \equiv\pm1\pmod{9}$.  This congruence is
    symplectic if and only if $\legendre{D}{p}=+1$.  For each
    such~$p$, $E^{-D}$ is the unique quadratic twist of~$E$ which is
    $p$-congruent to~$E$.

  \item
    For $D=4$: let $E'$ be the quartic twist of~$E$ by~$-4$, so that
    $E$ and $E'$ are isomorphic over~$M$ but not over~$\Q$.  Again,
    $E[p] \simeq E'[p]$ for all primes $p \geq 5$.  This congruence is
    symplectic if and only if $\legendre{2}{p}=+1$.  For each
    such~$p$, there are no quadratic twists of~$E$ which are
    $p$-congruent to~$E$.
  \end{enumerate}
\end{corollary}
\begin{proof}
Since $E$ has CM by~$M$ it follows from \cite[Propositions~1.14
  and~1.16]{Zywina} that the image of~$\rhobar_{E,p}$ is the full
normalizer of a Cartan subgroup.  (Here, the
condition~$p\equiv\pm1\pmod{9}$ is needed to ensure this when $D=3$.)
Moreover, $\PP \rhobar_{E,p} \simeq D_n \neq C_2 \times C_2$ (since $p
\geq 5$) and $\PP \rhobar_{E,p}(G_M) \simeq C_n$. Thus, for each
such~$p$, in the notation of Theorem~\ref{T:twist}, we have
$K(\sqrt{d})=M$, so $E[p]\cong E^{-D}[p]$ except for $D=4$ when
$E[p]\cong E'[p]$ with~$E'$ the quartic twist of~$E$ by~$-4$.  Since
the projective image is not~$C_2\times C_2$, in each case there are no
more $p$-congruences with curves isomorphic to~$E$ over quadratic
extensions (cf.~Theorem~\ref{T:twist}).

For the symplectic types, observe that $E$ and~$E'$ are isogenous via
an isogeny of degree~$2$ for $D=4$ (see Remark~\ref{R:2-3-isog}), while
for $D\not=4$ there is an isogeny of degree $D$ from~$E$ to~$E^{-D}$,
obtained by composing the twist isomorphism with the
endomorphism~$\sqrt{-D}$, which is defined over~$\Q$.
\end{proof}

\begin{remark}
In our computed data in Section~\ref{S:statistics}, we do not see
usually isomorphisms arising from CM curves as in this corollary. This
is because (apart from the CM cases where quartic or sextic twists
occur) all such mod~$p$ isomorphisms occur within an isogeny class,
and we have omitted these from consideration.
\end{remark}

Note that if $\PP\rhobar_{E,p} (G_K)$ is cyclic  then 
$\rhobar_{E,p} (G_K)$ is abelian. 
In particular, the smallest projective image occurring when $\rhobar_{E,p}$ is absolutely irreducible is 
$\PP\rhobar_{E,p} (G_K) \simeq D_2 \simeq C_2 \times C_2$. The next
result will be used below to determine the symplectic type of
congruences in the presence of this kind of image.

\begin{lemma}\label{L:C2xC2}
Let $p$ be an odd prime.  Let $N$ be a subgroup of~$\PGL_2(\Fp)$
isomorphic to $C_2\times C_2$, so that there are three subgroups~$C< N$
of index~$2$.

If $N\le\PSL_2(\Fp)$, then every such~$C$ is contained in a Cartan
subgroup and $N$ in its normalizer, and each~$C$ is split when
$p\equiv1\pmod4$ and non-split when $p\equiv3\pmod4$.

If $N\not\le\PSL_2(\Fp)$, then for $p\equiv1\pmod4$, one such
subgroup~$C$ is contained in a split Cartan subgroup while the other
two are contained in non-split Cartan subgroups; while if
$p\equiv3\pmod4$ then one~$C$ is non-split and the other two are
split.
\end{lemma}

\begin{proof}
  First note that in~$\PGL_2(\Fp)$, the condition of having zero
  trace is well-defined, and the determinant is also well-defined
  modulo squares.  Also, $\PSL_2(\Fp)$ is the subgroup of~$\PGL_2(\Fp)$ of elements with
  square determinant, which has index~$2$.  Hence either
  $N\le\PSL_2(\Fp)$, and all elements of~$N$ have square determinant,
  or $[N:N\cap\PSL_2(\Fp)]=2$, in which case exactly one of the
  elements of order~$2$ has square determinant.

  Examination of the characteristic polynomial shows that the elements
  of order~$2$ in~$\PGL_2(\Fp)$ are precisely those with trace zero,
  and these elements are split (having $2$ fixed points on
  $\PP^1(\Fp)$) if the determinant is minus a square, and non-split
  (having no fixed points) otherwise.  Hence when $p\equiv1\pmod4$,
  elements of order~$2$ are split if and only if they lie in
  $\PSL_2(\Fp)$, while for $p\equiv3\pmod4$ the reverse is the case.
  The result follows.
\end{proof}

The next theorem describes the symplectic type of congruences between
general quadratic twists.  Since the projective image is dihedral (by
Theorem~\ref{T:twist}), this is a situation where the local
methods from~\cite{FKSym} may not apply, as is illustrated by
Example~\ref{Ex:LocalFail7} below.  The first part of the theorem
(with $K=\Q$) is again already in~\cite[\S2]{Halberstadt-11nonsplit},
with essentially the same proof; we include it here in order to
include the second part, which describes a situation that cannot occur
over~$\Q$ except for very small primes. See Example~\ref{Ex:3twists}
below for an example with $p=3$.

\begin{theorem}\label{T:quadratic} 
Let $p$ be an odd prime. 
Let $E/K$ be an elliptic curve 
$p$-congruent to some quadratic twist~$E^d$.
Assume $E[p]$ is an absolutely irreducible~$G_K$-module, so that $j(E) \neq 0$ if $p=3$.

\begin{enumerate}

\item Let $C$ be the Cartan subgroup of~$\GL_2(\Fp)$
  associated to the extension $K(\sqrt{d})/K$ in
  Theorem~\ref{T:twist}.  Then the congruence is symplectic if and
  only if \emph{either} $C$ is split and $p \equiv 1 \pmod{4}$,
  \emph{or} $C$ is non-split and $p \equiv 3 \pmod{4}$.

\item When $\PP\rhobar_{E,p}\cong C_2\times C_2$, there are three
  different quadratic\footnote{when $j(E)=1728$ one of these
    is the quartic twist by~$-4$ so not in fact quadratic.}  twists
    of~$E$ which are $p$-congruent to~$E$.  If $\sqrt{p^*}\in K$ then
    all three congruences are symplectic.  Otherwise, one of the
    quadratic twists is by $K(\sqrt{p^*})$ and is symplectic while the
    other two are anti-symplectic.
\end{enumerate}
\end{theorem}
\begin{proof} 

(1) We will define a matrix~$A\in\GL_2(\Fp)$ satisfying~\eqref{E:PsiM}
and $\det(A)=-\delta$,
where $\delta \in \Fpstar$ is a square if $C$ is split and a non-square
if $C$ is non-split.   
Then, the map $E[p]\to E'[p]$ corresponding to~$A$
is then a
$G_K$-equivariant isomorphism which, 
by the discussion following~\eqref{E:PsiM}, 
is symplectic if and only if
$-\delta$ is a square.  From this, 
part~(1) follows by considering the four cases:
$\delta$ square/non-square and $p\equiv\pm1\pmod4$.

Let $H=G_{K(\sqrt{d})}$ be the index~$2$ subgroup cut out by 
the homomorphism~$\varepsilon_d:G_K\to\{\pm 1 \}$.

The projective image $\PP \rhobar_{E,p}(H)$ is a subgroup of $\PP(C) \subset \PGL_2(\Fp)$ and the latter is cyclic of even order~$p\pm1$
(depending on whether~$C$ is split or non-split). Hence $\PP(C)$ contains
a unique element of order~$2$. 
Let $A \in C \subset \GL_2(\Fp)$ be any lift of this element.  Then $A$ is
not scalar, while $A^2$ is scalar, so by Cayley-Hamilton we have
$\Tr(A)=0$ and, 
by the proof of Lemma~\ref{L:C2xC2}, we have that
$-\det(A) = \delta$ is square if and only if the Cartan subgroup~$C$ is split.

Since $A$ is central in $\rhobar_{E,p}(G_K)$ modulo scalars, for all
$g\in\rhobar_{E,p}(G_K)$ we have $AgA^{-1}=\lambda(g)g$ with a scalar
$\lambda(g)=\pm I$ (comparing determinants).  Now $\lambda(g)=I$ if and
only if $g$ commutes with~$A$ which---since $A$ is a non-scalar element
of the Cartan subgroup---is if and only if $g$ is itself in the Cartan
subgroup, that is, if and only if $g=\rhobar_{E,p}(\sigma)$ with
$\sigma\in H$, so  \eqref{E:PsiM} holds.

(2) Since the determinant of $\rhobar_{E,p}$ is the mod~$p$ cyclotomic
character, we have $\PP\rhobar_{E,p}(G_K)\subseteq \PSL_2(\Fp)$ if
and only if $\sqrt{p^*}\in K$.  When this is the case, by
Lemma~\ref{L:C2xC2} we see that each index~$2$ subgroup of the
projective image is split when $p\equiv1\pmod4$ and each is non-split
when $p\equiv3\pmod4$.  By part (1), it follows in both cases that the
congruences between~$E$ and each of the three twists are all
symplectic.

Now suppose that $\sqrt{p^*}\notin K$.  By Lemma~\ref{L:C2xC2} again,
exactly one of the three subgroups is split when $p\equiv1\pmod4$, and
exactly one is non-split when $p\equiv3\pmod4$, so in both cases
exactly one congruence is symplectic.  Moreover, since the subgroup
$C$ inducing the symplectic congruence is the unique one contained in
$\PSL_2(\Fp)$, the associated quadratic extension is~$K(\sqrt{p^*})$.
\end{proof}

\begin{example} 
\label{Ex:LocalFail7}
Consider the elliptic curve
\[ E : y^2 + y = x^3 - x^2 - 74988699621831x +  238006866237979285299, \]
which has conductor $ N_E = 7^2 \cdot 2381 \cdot
134177^2>2\cdot10^{15}$, so is not in the LMFDB database.  This
example was found using the explicit parametrization of curves for
which the image of the mod~$7$ Galois representation is contained in
the normalizer of a non-split Cartan subgroup; we verified by explicit
computation that the mod~7 image is equal to the full normalizer and that the Cartan subgroup cuts out the field $\Q(\sqrt{d})$ where $d = -7 \cdot 134177$. 

Consider~$E^d$, the quadratic twist
of~$E$ by~$d$. 
From Theorem~\ref{T:twist} we have that $E[7] \simeq E^d[7]$ as $G_\Q$-modules and part (1) of Theorem~\ref{T:quadratic} yields that $E[7]$ and $E^d[7]$ are
symplectically isomorphic (and not anti-symplectically
isomorphic). We note that the same conclusion can be obtained via the general method from Section~\ref{S:statistics} 
and, more interestingly, none of the local criteria
in~\cite{FKSym} applies for this case.
\end{example}

\begin{example} \label{Ex:3twists}
  For an example with projective image~$C_2\times C_2$, let $E$ be the
  elliptic curve with label~\lmfdbec{6534}{a}{1}, of conductor
  $6534=2\cdot3^3\cdot11^2$.  

  The image of the mod~$3$ Galois representation is the normalizer of
  the split Cartan subgroup, which
  is projectively isomorphic to $D_2 = C_2\times C_2$.  The three quadratic
  subfields of the projective $3$-division field are $\Q(\sqrt{-3})$,
  $\Q(\sqrt{-11})$, and~$\Q(\sqrt{33})$; the corresponding quadratic
  twist of~$E$ are $E^{-3}=\lmfdbec{6534}{v}{1}$,
  $E^{-11}=\lmfdbec{6534}{p}{1}$, and $E^{33}=\lmfdbec{6534}{h}{1}$
  respectively.  All four curves have isomorphic mod~$3$
  representations by Theorem~\ref{T:twist}, the isomorphism being antisymplectic between~$E$
  and~$E^{-11}$ and~$E^{33}$, and symplectic between $E$ and~$E^{-3}$,
  in accordance with part (2) of
  Theorem~\ref{T:quadratic}.
\end{example}

\subsection{Congruences between higher order twists} \label{SS:higherCong}
In this section we will study, under condition~{\bf (S)}, the congruences between an elliptic curve~$E/K$ and its quartic, cubic or sextic twists by~$u \in K$.

Note that if $u = -s^2$ then $u = -4(s/2)^2$ and the quartic twist
by~$u$ is obtained by composing the quartic twist by $-4$ with a
quadratic twist; similarly, if $u = -3s^2$ then $u = -27(s/3)^2$ and
the sextic twist by~$u$ is obtained by composing the quadratic twist
by $-27$ with a cubic twist. Since the quartic twist by~$-4$ and the
quadratic twist by~$-27$ correspond to isogenies (see
Remark~\ref{R:2-3-isog}) their effect on the symplectic type is
known; observe also that both these cases are covered by the theory in
Sections~\ref{SS:dihedral} and~\ref{SS:typeQuadratic}.  In view of
this, we fix the following natural assumptions for this and the next
section.

Let $p$ be an odd prime. 
Let $E/K$ be an elliptic curve 
with $j(E) = 0, 1728$ and let~$u \in K^*$.
\begin{itemize}
 \item If $j(E) = 1728$, assume also $K'=K(\sqrt{-1}) \neq K$, 
 $u \neq -1$ modulo squares;
 \item If $j(E) = 0$, assume also $K'=K(\sqrt{-3}) \neq K$, 
 $u \neq -3$ modulo squares and
$p\not=3$. 
\end{itemize}

\begin{lemma} \label{L:diagonalC}
Let~$E/K$ be as above. 
\begin{enumerate}
  \item
The projective image $\PP\rhobar_{E,p}(G_K)$ is dihedral $D_n$ for
some $n\ge2$, and the projective image of $G_{K'}$ is $C_n$.
\item
After extending scalars from~$\Fp$ to~$\F_{p^2}$ if necessary, there
is a basis for~$E[p]$ with respect to which for $\sigma\in G_{K'}$ we
have
\begin{equation}
 \label{E:diagonalC}
\rhobar_{E,p}(\sigma) = c(\sigma)\cdot\diag(1,\eps(\sigma))
\end{equation}
with $c(\sigma)\in\overline{\F}_p^*$ and
$\eps:G_{K'}\to\overline{\F}_p^*$ a character of exact order~$n$.
\end{enumerate}
\end{lemma}
\begin{proof}
Part (1) follows from standard facts on~CM curves plus our running
assumptions and~(2) follows by diagonalising the Cartan subgroup over $\Fp$ in
  the split case and over~$\F_{p^2}$ in the non-split case.
\end{proof}

\begin{lemma}\label{L:DiagonalPsi}
Let~$E/K$ and~$u\in K^*$ be as above.  For $m\in\{3,4,6\}$, let $\psi$
be the order~$m$ cocycle associated to~$u$, with values
in~$\Aut(E)\cong\calO^*$ where $\calO\cong\Z[\zeta_m]$.

After extending scalars to~$\Fbar_p$ and changing basis so that
\eqref{E:diagonalC} holds, the matrices $\Psi(\sigma)$
in~\eqref{E:PsiMatrix} become diagonal.  More precisely,
\[
\Psi(\sigma) = \diag(\eta(\sigma),\eta(\sigma)^{-1})
\]
where $\eta:G_K\to \overline{\F}_p^*$ is a map whose restriction to
$G_{K'}$ is a homomorphism of order exactly~$m$.
\end{lemma}
\begin{proof}
Recall that we have fixed an isomorphism 
$\End(E)\cong\calO$ and
that $\rhobar_{E,p}(G_{K'}) \subset C$ for some Cartan subgroup~$C \subset \GL_2(\Fp)$.
Hence $\calO$ acts on $E[p]$ 
via matrices which are in~$C$ since the endomorphisms
are all defined over~$K'$ and so their action commutes with that
of~$G_{K'}$, and the action of~$G_{K'}$ is not scalar. Note that the basis giving~\eqref{E:diagonalC} diagonalizes the whole of~$C$, therefore~$\Psi(\sigma)$ is diagonal. Finally,  
an endomorphism~$\alpha$ acts via a matrix of determinant~$\deg(\alpha)$ so $\calO^*$ acts via diagonal matrices of determinant~$1$ as stated.

The map~$\eta$ is obtained as follows.  In the split case
$p\calO=\frp\overline{\frp}$ and $\eta(\sigma)$ is the image of
$\psi(\sigma)$ under the isomorphism
$\End(E)\cong\calO$ followed by the reduction
$\calO\to\calO/\frp\cong\Fp$ which induces a reduction homomorphism
$\eta:\calO^*\to\Fpstar$.  Interchanging~$\frp$ and~$\overline{\frp}$
has the effect of replacing~$\eta$ by its inverse; we make an
arbitrary but fixed choice.

In the non-split case, $p\calO=\frp$ and 
$\eta(\sigma)$ is the image
of $\psi(\sigma)$ under the isomorphism followed by the reduction
$\calO\to\calO/\frp\cong\F_{p^2}$. 
This induces a 
homomorphism $\eta: \calO^*\to\F_{p^2}^*$ for which there are two
choices since we may compose with the nontrivial
automorphism of~$\F_{p^2}$.

By definition, the map~$\eta$ becomes a homomorphism of order~$m$ when restricted to~$G_{K'}$ because this is the case for~$\psi$ by cases (ii) and (iii) 
of Section~\ref{SS:GalRepTwist}, given that we have excluded the special quartic twist by~$u=-4$.
\end{proof}

\begin{lemma}\label{L:EpsEta}
  Keeping the notation of Lemmas~\ref{L:diagonalC}
  and~\ref{L:DiagonalPsi}, let $E'$ be the order~$m$ twist of~$E$
  by~$u$ with~$m \in \{3,4,6\}$, and suppose also that $E$ and $E'$
  are $p$-congruent.  Then $\eps(\sigma)=\eta(\sigma)$ for all
  $\sigma\in G_{K'}$, and $\PP \rhobar_{E,p}(G_K) \simeq D_n$ where
  $n=m$.
\end{lemma}
\begin{proof}
  Let $A\in\GL_2(\Fp)$ be the matrix of an isomorphism~$E[p]\to
  E'[p]$.  For $\sigma\in G_{K'}$ we have, from \eqref{E:PsiM} and
  after cancelling the scalar factor~$c(\sigma)$ from each side,
\[
A\ \diag(1,\eps(\sigma))\ A^{-1} = \diag(1,\eps(\sigma))\; \diag(\eta(\sigma),\eta(\sigma)^{-1}).
\]
Now since $A$ conjugates a non-scalar diagonal matrix into another
diagonal matrix, it is either itself diagonal or is anti-diagonal.
But $A$ cannot be diagonal since that would imply $\eta(\sigma)=1$ for
all $\sigma\in G_{K'}$ (which is not the case because we have excluded the special quartic twist), so $A$ is anti-diagonal.  Therefore, conjugating any
diagonal matrix by~$A$ just interchanges the two diagonal entries, so
the previous equation becomes
\[
\diag(\eps(\sigma),1) = \diag(1,\eps(\sigma))\; \diag(\eta(\sigma),\eta(\sigma)^{-1}).
\]
Hence $\eps(\sigma)=\eta(\sigma)$ for all $\sigma\in G_{K'}$, thus $n=m$ and the
statement of the lemma follows.
\end{proof}

The preceding lemma shows that, for $n\in\{3,4,6\}$, a necessary
condition for the existence of a $p$-congruence between~$E$ with a
twist~$E'$ of order~$n$ is that the projective mod~$p$ image is
isomorphic to $D_n$.  We next show that this condition is also
sufficient.  First we have an elementary lemma.

\begin{lemma}
Let $n\in\{3,4,6\}$, and let $K$ be a number field not containing the
$n$th roots of unity.  Let $F/K$ be a Galois extension with
$\Gal(F/K)\cong D_n$, and assume that the subfield $K'$ of $K$ fixed
by the unique cyclic subgroup of order~$n$ is $K'=K(\sqrt{-1})$ if
$n=4$ and $K'=K(\sqrt{-3})$ if $n=3,6$.

Then there exists $u\in K$ such that $F$ is the splitting field of~$X^n-u$.
\end{lemma}

\def\rnu{\root n\of u} \def\rnub{\root n\of {\overline{u}}}

\begin{proof}\label{L:Dn}
Write $\Gal(F/K)=\left<\sigma,\tau\mid\sigma^n=\tau^2=1,
\tau\sigma\tau=\sigma^{-1}\right>$.  The fixed field of $\sigma$
is~$K'$.  Since $K'$ contains the $n$th roots of unity, by Kummer
Theory, $F=K'(\rnu)$ for some $u\in K'$.  Now $\sigma(\rnu)=\zeta\rnu$
with $\zeta$ a primitive $n$th root of unity, and either $\tau(u)=u$
or $\tau(u)=\overline{u}$ (the $K'/K$-conjugate of~$u$).  In the first
case, $u\in K$ and the result follows.

Suppose that $\tau(u)=\overline{u}\not=u$.  Since
$(\tau(\rnu))^n=\tau(u)=\overline{u}$, we may set $\rnub=\tau(\rnu)$.
Using $\sigma=\tau\sigma^{-1}\tau$ we find that
$\sigma(\rnub)=\zeta\rnub$.  Hence $v=\rnub/\rnu\in K'$, so
$\rnub=\rnu v$;  applying~$\tau$ gives
\[
\rnu = \rnub\;\overline{v} = \rnu v\overline{v},
\]
so $v\overline{v}=1$.  Then $v(1+\overline{v})=1+v$ and it follows
that $u_1=u(1+v)^n\in K$.  Replacing $u$ by $u_1$ completes the
argument.
\end{proof}

\begin{theorem} \label{T:higherTwists}
Let $E/K$ and $p$ be as above. Suppose that $\PP \rhobar_{E,p}(G_K) \simeq D_n$ where $n=4$ if
$j(E)=1728$ and $n\in\{3,6\}$ if $j(E)=0$.
Then $E$ is $p$-congruent to an $n$-twist.  

Moreover, for $n=3$ there is a
unique such twist, while for $n=4$ there are two which are
$2$-isogenous to each other and for $n=6$ there are two which are
$3$-isogenous to each other.
\end{theorem}

\begin{proof} 
The projective $p$-division field of~$E$ is the Galois extension~$F/K$
fixed by $\PP \rhobar_{E,p}$ with $\Gal(F/K) \simeq D_n$. By
Lemma~\ref{L:Dn} we can write it as $F=K'(\root n\of u)$ for some
$u\in K^*$, which is unique up to replacing $u$ by~$u^{-1}$ and
multiplication by an element of~$K^*\cap(K'^*)^n$.  (Here we use that
fact that~$\Aut(C_n)\cong\{\pm1\}$ for $n\in\{3,4,6\}$.)

Keeping the notations from the previous lemmas, to prove the first
part of the theorem we will construct a matrix~$A\in\GL_2(\Fp)$
giving the $G_K$-isomorphism $E[p]\to E'[p]$.

Indeed, the character $\eps:G_{K'} \to
\overline{\F}_p^*$ has exact order~$n$ and cuts out the extension
$F/K'$.
Now let~$\psi:G_K\to \mu_n\subseteq\calO^*$ be the cocycle associated
to~$u$, namely
\[
\sigma\mapsto \psi(\sigma) = \sigma(\root n\of u)/\root n\of u.
\]
The restriction of~$\psi$ to~$G_{K'}$ is a character of order~$n$
which cuts out the same extension $F/K'$, so it is either~$\eps$ or~$\eps^{-1}$.  Replacing $u$ by $u^{-1}$ if necessary, we may assume that
$\left.\psi\right|_{G_{K'}} = \eps$.

Now let $E'$ be the order~$n$ twist of~$E$ by~$u$.  As before, we have
$\Psi(\sigma)=\diag(\eta(\sigma),\eta^{-1}(\sigma))$ where
$\left.\eta\right|_{G_{K'}} = \left.\psi\right|_{G_{K'}} = \eps$ by Lemma~\ref{L:EpsEta}.
Thus, for $\sigma\in G_{K'}$, we have
\[
\Psi(\sigma) = \diag(\eps(\sigma),\eps(\sigma)^{-1}),
\]
and so, the same computation as in the proof of Lemma~\ref{L:EpsEta} 
gives that, for any anti-diagonal matrix~$A$
and all $\sigma\in G_{K'}$, we have
\[
A\ \diag(1,\eps(\sigma))\ A^{-1} = \diag(\eps(\sigma),1) =
\diag(1,\eps(\sigma)) \Psi(\sigma),
\]
that is~\eqref{E:PsiM} holds for $\sigma\in G_{K'}$.  To show
that~\eqref{E:PsiM} holds for all $\sigma\in G_{K}$, it suffices,
since both sides of~\eqref{E:PsiM} are homomorphisms $G_K\to
\GL_2(\Fp)$, to do so for a single element~$\tau\in G_K\setminus
G_{K'}$.

Write $\Psi(\tau)=\diag(w,w^{-1})$ with $w\in\F_{p^2}^*$.  Define
\[
A=\rhobar_{E,p}(\tau)\diag(w,1)=\diag(1,w)\rhobar_{E,p}(\tau),
\]
the second equality following since~$\rhobar_{E,p}(\tau)$ is
antidiagonal.  Then $\rhobar_{E,p}(\tau)A^{-1}=\diag(1,w^{-1})$, so
\[
A \rhobar_{E,p}(\tau) A^{-1} = \rhobar_{E,p}(\tau)\Psi(\tau)
\]
as required.

In the split Cartan case we are done, as the diagonalization of~$C$
occurs in~$\GL_2(\Fp)$ and so~\eqref{E:PsiM} holds with $A \in
\GL_2(\F_{p})$.  In the non-split case, we have shown that
equation~\eqref{E:PsiM} holds over $\GL_2(\F_{p^2})$. By undoing the
initial change of coordinates we obtain that \eqref{E:PsiM} holds
in~$\GL_2(\Fp)$.

We will now prove the second statement. First, recall from the first paragraph that $u \in K^*$ is unique up to inverse and
multiplication by an element
in~$K^*\cap(K'^*)^n$. Therefore, up to $n$th powers in~$K^*$,  
we have either four or two
possible choices for~$u$, namely
\begin{itemize}
  \item $\{u, u^{-1}, -4u, -4u^{-1}\}$ if $n=4$;
  \item $\{u, u^{-1}, -27u, -27u^{-1}\}$ if $n=6$;
  \item $\{u, u^{-1}\}$ if $n=3$.
\end{itemize}

This follows from the observation that the natural map $K^*/(K^*)^n
\to (K'^*)/(K'^*)^n$ is injective when $n=3$, and has kernel $\{1,-4\}$
for $n=4$ and $\{1,-27\}$ for $n=6$.

Secondly, observe that the construction in the first part of the proof
only works for exactly one out of each inverse pair~$u,u^{-1}$ mod
$n$th powers, so we have two quartic or sextic twists when $n=4$ or
$n=6$ respectively, and just one cubic twist when $n=3$.  The two
quartic twists differ by the quartic twist by~$-4$, so by a
$2$-isogeny, while the two sextic twists differ by the sextic twist
by~$-27$, so by a $3$-isogeny (see Remark~\ref{R:2-3-isog}).
\end{proof}

See Theorem~\ref{T:higherTwistsQ} for applications of this theorem,
including the determination of the symplectic type of the congruences
using Theorem~\ref{T:typeHigher} below.

\subsection{The symplectic type of congruences between higher order twists}
\label{SS:typeHigher}
We have classified above, under condition~{\bf (S)}, exactly when congruences between twists occur. To complete this part of our study we are left to describe the symplectic type of congruences between higher order twists. 
This is given by the following result. 

\begin{theorem}\label{T:typeHigher}
  Let $p$ be an odd prime, $K$ a number field, $E/K$ an elliptic
  curve, and $u\in K^*$.  Assume that
  \begin{itemize}
    \item \emph{either:} $j(E)=1728$, $\sqrt{-1}\notin K$, $u\not=\pm1$
      modulo squares, and $n=4$;
    \item \emph{or:} $j(E)=0$, $\sqrt{-3}\notin K$, $u\not=1,-3$
      modulo squares, and $n=3$ or~$6$.
  \end{itemize}

 Let $E'/K$ be the order~$n$ twist of~$E/K$ by~$u$.
 Suppose that $E$ and~$E'$ are $p$-congruent.
\begin{enumerate}
\item If $\sqrt{p^*}\in K$ then $E[p]$ and~$E'[p]$ are symplectically isomorphic
and $p\equiv\pm1\pmod{2n}$.
\item Assume $\sqrt{p^*}\notin K$. Then:
\begin{enumerate}
\item if $n=3$ then $E[p]$ and~$E'[p]$ are anti-symplectically isomorphic;
\item if $n=4$ then $E[p]$ and~$E'[p]$ are symplectically isomorphic if and only if $\sqrt{up^*}\in K$, and the congruence with the quartic twist by~$-4u$ has the opposite symplectic type; 
moreover, $p\equiv\pm3\pmod8$;
\item if $n=6$ then $E[p]$ and~$E'[p]$ are symplectically isomorphic if and only if
  $\sqrt{up^*}\in K$, and then the sextic twist by~$-27u$ is
  antisymplectic; moreover, $p\equiv\pm5\pmod{12}$.
\end{enumerate}
\end{enumerate}
\end{theorem}

\begin{proof}
This proof builds on the proof of Theorem~\ref{T:higherTwists}.
Indeed, we have chosen $\tau\in G_K$ to be such that it fixes~$\root n\of u$ and 
is non-trivial when restricted to~$K'$ and 
we let $A = \rhobar_{E,p}(\tau)$. We have shown that~$A$ satisfies~\eqref{E:PsiM} and so, by the discussion following~\eqref{E:PsiM}, 
the symplectic type of the congruence is given by the square class of~$\det A$. 

We also know that the projective $p$-division field is $F = K'(\root n \of u)$
and $\Gal(F/K) \simeq D_n$.

We have $\det(A)=\det\rhobar_{E,p}(\tau)=\chi_p(\tau)$ where $\chi_p$ denotes the mod~$p$ cyclotomic character.  Since 
the $p$-division field of~$E$ contains the $p$-th roots of unity, and
the projective $p$-division field contains~$\sqrt{p^*}$, the
congruence is symplectic 
if and only if $\tau$ fixes~$\sqrt{p^*}$.
Clearly, when $\sqrt{p^*}\in K$ the congruence is symplectic, proving (1) except for the congruence condition.

Suppose now that $\sqrt{p^*}\notin K$.
We divide into cases:

(a) Suppose $n=3$. 
We have $\Gal(F/K) \simeq D_3$ and so $K'$ is the unique quadratic subfield of~$F$, therefore $\sqrt{p^*} \in K'$. Since~$\tau$ acts non-trivially on~$K'$ the congruence is anti-symplectic.

(b) Suppose $n=4$. We have $\Gal(F/K) \simeq D_4$ and there are exactly three quadratic sub-extensions of~$F/K$.
Furthermore, the fields 
$K'=K(\sqrt{-1})$, $K(\sqrt{p^*})$, $K(\sqrt{u})$ 
are quadratic extensions of~$K$ satisfying
$K' \neq K(\sqrt{p^*})$, $K(\sqrt{u})$. 

(b1) Suppose $K(\sqrt{p^*})=K(\sqrt{u})$.  By definition, $\tau$
fixes~$\sqrt{u}$, thus it also fixes~$\sqrt{p^*}$, and the congruence
is symplectic; note that in this case~$\sqrt{up^*} \in K$.
 
(b2) Suppose $K(\sqrt{p^*}) \neq K(\sqrt{u})$. Then 
$K(\sqrt{p^*})=K(\sqrt{-4u})$ and 
since $\tau(\sqrt{-4u}) = - \sqrt{-4u}$
the congruence is anti-symplectic;
note that in this case~$\sqrt{up^*} \not\in K$. 

Recall from Theorem~\ref{T:higherTwists} that there is also a
$p$-congruence between~$E$ and its quartic twist by~$-4u$. Applying
the previous argument to this congruence gives that it is symplectic
if and only if $\sqrt{-4up^*} \in K$. Thus the two congruences are of
the same type if and only if~$\sqrt{-1} \in K$, which is not the case
by assumption.

(c) Suppose $n=6$. We have $\Gal(F/K) \simeq D_6$ and again there are exactly three quadratic sub-extensions of~$F/K$.
Furthermore, the fields 
$K'=K(\sqrt{-3})$, $K(\sqrt{p^*})$, $K(\sqrt{u})$ 
are quadratic extensions of~$K$ satisfying
$K' \neq K(\sqrt{p^*})$, $K(\sqrt{u})$. 
The rest of the argument follows similarly to case~(b).

For the congruence conditions on~$p$, which are only non-trivial when
$n=4$ or~$6$ (since $p>3$ implies $p\equiv\pm1\pmod6$) recall that
when $n=4$ or~$6$ the two quartic (respectively, sextic) twists are
$2$-isogenous (respectively $3$-isogenous) to each other.  When
$\sqrt{p^*}\in K$ these isogenies induce symplectic congruences, since
both the twists of~$E$ are symplectically congruent.  By the isogeny
criterion, this implies that $2$ (respectively~$3$) is a quadratic
residue, so $p\equiv\pm1\pmod{8}$ (respectively, mod~$12$).  When
$\sqrt{p^*}\notin K$ these isogenies induce anti-symplectic
congruences, so $p\not\equiv\pm1\pmod{8}$ (respectively, mod~$12$).
\end{proof}

We end this section by proving Theorem~\ref{T:higherTwistsQ} from the
Introduction, as an illustration of how Theorems~\ref{T:higherTwists}
and~\ref{T:typeHigher} may be applied.  These include all
possibilities for quartic and sextic twist congruences over~$\Q$ with
$p=3$ or~$5$ in the quartic case and $p=5$ or~$7$ in the sextic case.

\begin{proof}[Proof of Theorem~\ref{T:higherTwistsQ}] Let $K=\Q$.  We
  consider curves with $j$-invariant~$1728$ and~$0$ in turn.

\subsubsection*{$j=1728$}
  First consider the curves $E_a := E_{a,0}:\ Y^2=X^3+aX$, for $a \in
  \Q^*$, which have $j$-invariant~$1728$.

  By Proposition~1.14(ii)
  in~\cite{Zywina}, the mod~$3$ projective image is the normaliser of
  a non-split Cartan subgroup, isomorphic to~$D_4$.  The projective
  division field~$F$ is obtained by adjoining the roots of the
  $3$-division polynomial $3X^4+6aX^2-a^2$, and
  $F=\Q(\sqrt{-1},\root4\of u)$ where~$u=-a^2/3$ and it is not hard to
  see that $a/\sqrt{-3}$ is not a square in $\Q(\sqrt{-1},\sqrt{-3})$.

Following the proof of Theorem~\ref{T:higherTwists}, one checks
that up to $4$th powers, the subgroup of $\Q^*/(\Q^*)^4$ which become
$4$th powers in~$F$ is generated by~$u$ and~$-4$: the two maps
$\Q^*/(\Q^*)^4 \to \Q(\sqrt{-1})^*/(\Q(\sqrt{-1})^*)^4$ and
$\Q(\sqrt{-1})^*/(\Q(\sqrt{-1})^*)^4 \to F^*/(F^*)^4$ have kernels
generated by~$-4$ and~$u$ respectively.  Hence, from
Theorem~\ref{T:higherTwists}, we expect a $3$-congruence between
$E_a$ and either $E_{-1/(3a)}$ (twisting by~$u$) or $E_{-3/a}$
(twisting by~$u^{-1}=-3/a^2$).  An explicit computation shows that
only the first holds.

Finally, we have $p^* = -3$ and $u=-a^2/3$, hence $\sqrt{p^*u} = a \in
\Q$ and the congruence is symplectic by Theorem~\ref{T:typeHigher}.
Moreover, there is a congruence with the quartic twist by $-4u=4a^2/3$
which is antisymplectic.

In summary, $E_a$ is symplectically $3$-congruent to $E_{-1/3a}$ and
anti-symplectically $3$-congruent to both $E_{-4a}$ and $E_{4/3a}$.
Note that in passing from the special case of the $3$-congruence
between $E_1$ and $E_{-1/3}$ to the general case of the congruence
between $E_{a}$ and $E_{-1/3a}$, we apply the quartic twist by~$a$ to
the first curve, but the inverse twist (by~$a^{-1}$) to the second.

Now let $p=5$.  The projective image is the normaliser of a split
Cartan subgroup, isomorphic to~$D_4$, by Proposition 1.14(i)
of~\cite{Zywina}.  One can check that up to $4$th powers, the subgroup
of $\Q^*/(\Q^*)^4$ which become $4$th powers in the projective
$5$-division field of~$E_a$ is generated by~$5/a^2$ and~$-4$.  Hence
we expect $5$-congruences between $E_a$ and either $E_{5/a}$ (twisting
by~$u=5/a^2$) or $E_{a^3/5}$ (twisting by~$u=a^2/5$).  Only the first
holds (by a computation similar to the previous example, though a
little simpler since we are in the split case so do not need to extend
scalars).  Since $u=5/a^2$ is $5$ times a square, and~$p^*=5$, the
congruence is symplectic.  There is also a congruence with the quartic
twist by $-4u=-20/a^2$, which is antisymplectic.

\subsubsection*{$j=0$}
Next consider the family of curves $E_b:=E_{0,b}:\ Y^2=X^3+b$, which
have $j$-invariant~$0$.

For $p=5$, we can apply Proposition 1.14(iv) of~\cite{Zywina} to see
that the projective image is the normaliser of a non-split Cartan
subgroup, isomorphic to~$D_6$, unless $b/10$ is a cube, in which case
the projective image is~$D_2$.  In the case where $b/10$ is not a
cube, we expect a $5$-congruence between $E_b$ and $E_{bu}$ where
$u\equiv5$ (modulo squares) and $u\equiv b/10$ or~$10/b$ (modulo
cubes), since one may check that $b/10$ is a cube in $\Q(E_b[5])$.
Hence, modulo $6$th powers, we have either $u\equiv 4/5b^2$ or
$u\equiv 5b^2/4$.  The first works, hence $E_b$ and $E_{4/(5b)}$ are
$5$-congruent.  This congruence is symplectic; composing with the
$3$-isogeny we also have an anti-symplectic congruence between $E_b$
and $E_{-108/(5b)}$.  In case $b/10$ is a cube, the sextic twist
by~$b/10$ is a quadratic twist, and we have three different
$5$-congruent quadratic twists, as expected when the projective image
is~$D_2$.

When $p=7$, by Proposition 1.14(iii) of~\cite{Zywina}, the projective
image is the normaliser of a split Cartan subgroup, isomorphic
to~$D_6$, unless $7b/2$ is a cube, in which case the projective image
is again~$D_2$.  In the general case, we expect a $7$-congruence
between $E_b$ and $E_{bu}$ where $u\equiv-7$ (modulo squares) and
$u\equiv 7b/2$ or~$2/(7b)$ (modulo cubes), since one may check that
$7b/2$ is a cube in $\Q(E_b[7])$.  Hence, modulo $6$th powers, we have
either $u\equiv -28/b^2$ or $u\equiv -b^2/28$.  The first works, hence
$E_b$ and $E_{-28/b}$ are $7$-congruent.  This congruence is
symplectic; composing with the $3$-isogeny we also have an
anti-symplectic congruence between $E_b$ and $E_{756/b}$. In case
$7b/2$ is a cube, the sextic twist by~$7b/2$ is a quadratic twist, and
we have three different $7$-congruent quadratic twists, as expected
when the projective image is~$D_2$.
\end{proof}

\section{Finding congruences and their symplectic type}\label{S:statistics}

In this section we discuss our systematic study of mod~$p$ congruences
between elliptic curves in the LMFDB database.  As of September 2019, this
database contains all elliptic curves defined over~$\Q$ of
conductor~$N\le\numprint{500000}$, as computed by the first author
using the methods of~\cite{AMEC}; there are $\numprint{3064704}$
curves, in $\numprint{2164259}$ isogeny classes.

Recall first that isogenous curves have mod~$p$ representations which
are isomorphic up to semisimplification, and actually isomorphic if
the degree of the isogeny is not divisible by~$p$.  Secondly, two
representations have isomorphic semisimplification if and only if they
have the same traces, so that we can test this condition by testing
whether
\[ a_{\ell}(E)\equiv a_{\ell}(E')\pmod{p}
\quad \text{for all primes } \ell \nmid pNN',
\] 
where $N$ and $N'$ are the conductors of $E$
and~$E'$ respectively.  This test can very quickly establish rigorously
that two curves do \emph{not} have isomorphic $p$-torsion up to
semisimplification, by finding  a single
prime~$\ell$ such that $a_{\ell}(E)\not\equiv a_{\ell}(E')\pmod{p}$.
Moreover, it is possible to prove that two curves have isomorphic $p$-torsion up to
semisimplification using this test for a finite number of
primes~$\ell$, as we explain in Step 2 below.

We divide our procedure to determine all mod~$p$ congruences between
non-isogenous curves, and their symplectic type, for a fixed
prime~$p$, into five steps.  Note that, as remarked in the
Introduction, Condition (S) is satisfied for $p\ge7$ for all elliptic
curves defined over~$\Q$.  We first outline the steps, and then
consider each in detail in the following subsections.
\begin{enumerate}[1.]
\item Partition the set of isogeny classes of elliptic curves in the LMFDB
  into subsets~$S$, such that whenever two curves have mod~$p$
  representations with isomorphic semisimplifications, their isogeny
  classes belong to the same subset~$S$, but not necessarily
  conversely.
\item For each subset~$S$ prove that the curves in each isogeny class in~$S$ really
  do have isomorphic mod~$p$ representations up to
  semisimplification, if necessary further partitioning the subsets.
  Discard all ``trivial'' subsets of size~$1$.
\item Separate the remaining subsets resulting from the previous
  step into those which have irreducible mod~$p$ representations and
  the reducible ones.
\item For each irreducible subset~$S$, and each pair of isogeny
  classes in~$S$, pick curves $E$ and~$E'$, one from each class in the
  pair; determine the symplectic type of the triple~$(E,E',p)$; then
  use the isogeny criterion to partition the set of all the curves in
  all the isogeny classes in~$S$ into one or two parts such that curves
  in the same part are symplectically isomorphic while those in
  different parts are antisymplectically isomorphic.
\item For each reducible subset~$S$, determine whether, for each pair
  $E$, $E'$ chosen as in Step 4, there is an isomorphism between
  $E[p]$ and $E'[p]$ and not just between their semisimplifications,
  if necessary replacing $E'$ with the curve $p$-isogenous to it.  If
  not, this means that $E[p]$ and $E'[p]$ are not in fact isomorphic.
  Thus we further partition each reducible set $S$ into subsets of
  isogeny classes of curves whose mod~$p$ representations are actually
  isomorphic, not just up to semisimplification.  For each of these
  new subsets, if nontrivial, proceed as in Step~4.

\end{enumerate}
Next we will explain each step in further detail.  For the first three
steps, $p$ is arbitrary, and we have carried these steps out for $7\le
p\le97$.  According to Theorem~\ref{T:cong19}, no congruences
(other than those induced by isogenies) exist for larger~$p$.  For the
last two steps, we restrict to $p=7$ which is the most interesting
case, as remarked in the Introduction.

\subsection{Sieving}
In order that $E[p]\cong E'[p]$ up to semisimplification, it is
necessary and sufficient that for all primes~$\ell$ not dividing
$pNN'$ we have $a_{\ell}(E)\equiv a_{\ell}(E')\pmod{p}$.  In this step
we may take one curve from each isogeny class, since isogenous curves
have the same traces~$a_\ell$, and have mod~$p$ representations with
isomorphic semisimplifications.

Fix an integer~$B\ge1$.  Let $\calL_B$ be the set of the $B$ smallest
primes greater than $\numprint{500000}$. All curves in the database
have good reduction at each prime in~$\calL_B$.  Assume also that
$p\notin\calL_B$. Hence a necessary condition for
two curves $E$ and $E'$ in the database to be congruent mod~$p$ is
that $a_{\ell}(E)\equiv a_{\ell}(E')\pmod{p}$ for all~$\ell\in\calL_B$.

To each curve $E$ in the database we assign a ``hash value'' which is
a simple function of the set $\{a_{\ell}(E)\pmod{p}\mid
\ell\in\calL_B\}$.  For example we may enumerate
$\calL_B=\{\ell_0,\ell_1,\dots,\ell_{B-1}\}$ and use the integer value
$\sum_{i=0}^{B-1}\overline{a}_{\ell_i}(E)p^i$, where for $a\in\Z$,
$\overline{a}$ denotes the reduction of $a$ mod~$p$ which lies in
$\{0,1,\dots,p-1\}$.  Curves whose mod~$p$ representations are
isomorphic up to semisimplification will have the same hash, and we
may hope that clashes will be rare if $B$ is not too small.

We proceed to compute this hash value for one curve in each isogeny
class in the database, recording the curve's label in a list indexed
by the different hash values encountered.  At the end of this step we
can easily form a partition of the set of isogeny classes by taking
these lists for each hash value.  We then discard any such lists which
are singletons.  Using $B=40$, this process takes approximately 40
minutes for a single prime~$p$.  Note, however, that as most of the
computation time taken is in computing $a_{\ell}(E)$ for all
curves~$E$ (up to isogeny), it is more efficient to compute the hash
values for several primes in parallel.

{\bf Example.} After carrying out this step for $7\le p\le17$, using
$B=50$, we find: $\numprint{23735}$ nontrivial subsets for $p=7$;
$731$ for $p=11$; $177$ for $p=13$; and~$8$ for $p=17$.  There are no
nontrivial subsets for any primes~$p$ with $19\le p\le97$, so we can
immediately conclude that there are no congruences in the database
between non-isogenous curves modulo any prime in this range.  See also
Theorem~\ref{T:cong19}.

\subsection{Proving isomorphism up to semisimplification}

For each pair of isogeny classes within one subset obtained in the
previous step, we use a criterion of Kraus--Oesterl\'e (see
\cite[Proposition~4]{KO}), based on the Sturm bound and hence on the
modularity of elliptic curves over~$\Q$, to either prove isomorphism
up to semisimplification, or reveal a ``false positive''.  The latter
would happen if two curves which are not congruent mod~$p$ have traces
of Frobenius~$a_{\ell}$ which are congruent modulo~$p$ for
all~$\ell\in\calL_B$.

{\bf Example (continued).}  For $7\le p\le17$ we find no such false
positives, so the curves within each subset do have mod~$p$
representations which are genuinely isomorphic up to
semisimplification.

\begin{remark}
  To avoid false positives, it is necessary to use a value of $B$
  which is large enough.  In our initial computations with conductor
  bound~$\numprint{400000}$ we initially used 30 primes
  above~$\numprint{400000}$. But the curves with labels
  \lmfdbec{25921}{a}{1} and \lmfdbec{78400}{gw}{1} have traces
  $a_{\ell}$ which are \emph{equal for all~$\ell\in\calL_{35}$}, that
  is, for all~$\ell$ with $400000\le \ell<400457$ (though not for
  $\ell=400457$).  These curves have CM by the order of
  discriminant~$-7$, and are quadratic twists by~$230$; both have
  $a_\ell=0$ for all $\ell\equiv3,5,6\pmod{7}$, and $230$ is a
  quadratic residue modulo all other primes in~$\calL_{35}$. In our
  first computational runs (with $N\le\numprint{400000}$), we used
  $B=30$ and discovered this pair of curves giving rise to a false
  positive for every~$p$.
\end{remark}

The sizes of the subsets of isogeny classes we find after the first
two steps are as follows: for $p=7$ the~$\numprint{23735}$ subsets
have sizes between~$2$ and~$80$; for $p=11$, $p=13$ and $p=17$ they
all have size~$2$.

\subsection{Testing reducibility}
For each set of isogeny classes of curves obtained in the previous
step, we next determine whether the curves in the set have irreducible
or reducible mod~$p$ representations.  To do this we apply a standard
test of whether an elliptic curve admits a rational $p$-isogeny.  For
the curves in the database this information is already known.

{\bf Example (continued).} For $p=7$, of the $\numprint{23735}$
nontrivial sets from Step~2, we find that $\numprint{23448}$ are
irreducible, {\it i.e.}\@ consist of curves whose mod~$7$
representations are irreducible, while $287$ are reducible.

The irreducible sets have size at most~$5$.  In detail, there are
$\numprint{21653}$ sets of size~2; $\numprint{1502}$ sets of size~3;
283 sets of size~4; and 10 sets of size~5.

The reducible sets have size up to~$80$.  In Step~5 below we will further
partition these sets after testing whether the curves are actually congruent
mod~7 (not just up to semisimplification), after which the largest subset has
only~$4$ isogeny classes.


For $p=11, 13$, and~$17$, all the nontrivial subsets are of size~$2$,
and all are irreducible.

\subsection{Distinguishing symplectic from antisymplectic: irreducible case}
After the previous step we have a collection of sets of isogeny
classes, such that for each pair of curves $E$, $E'$ taken from
isogeny classes in each set, the $G_{\Q}$-modules $E[p]$ and $E'[p]$
are isomorphic and irreducible. Moreover, from
Proposition~\ref{P:conditionS} we know that all isomorphisms $\phi : E[p]
\simeq E'[p]$ have the same symplectic type. We wish to determine
whether this type is symplectic or anti-symplectic.  We may assume
that $E$ and~$E'$ are not isogenous, as otherwise we may simply apply
the isogeny criterion.

The local criteria of \cite{FKSym} suffice to determine the symplectic
type for all the mod~$p$ congruences found in the database for $p=7$
and $p=11$ (and also for $p=13, 17$), but this does not have to be
the case as discussed in Section~\ref{S:motivation} (see
\cite[Proposition~16]{FKSym} for an example with $p=3$ where the local
methods fail).  Therefore, we will now describe a procedure, using the
modular curves~$X_E(7)$, to obtain a method that works in all cases.
We will use the modular parametrizations and explicit formulae of
Kraus--Halberstadt~\cite{Halberstadt-Kraus-XE7},
Poonen--Schaefer--Stoll~\cite{PSS}, and as extended and completed by
Fisher~\cite{Fisher}.

In~\cite{Halberstadt-Kraus-XE7}, Halberstadt and Kraus give an
explicit model for the modular curve $X_E(7)$, for any elliptic curve
$E$ defined over a field~$K$ of characteristic not equal to $2$, $3$
or~$7$. Recall that the $K$-rational points on $X_E(7)$ parametrize
pairs $(E',\phi)$ where $E'$ is an elliptic curve defined over~$K$ and
$\phi:E[7]\to E'[7]$ is a symplectic isomorphism of $G_K$-modules; we
identify two such isomorphisms~$\phi$ when one is a scalar multiple of
the other.

The model for $X_E(7)$ given in \cite{Halberstadt-Kraus-XE7} is a
plane quartic curve, a twist of the classical Klein quartic~$X(7)$,
given by an explicit ternary quartic form~$F_{a,b}(X,Y,Z)$ in
$\Z[a,b][X,Y,Z]$ where $E$ has equation $Y^2=X^3+aX+b$.  The $24$
flexes on $X_E(7)$ are the cusps, that is, they are the poles of the
rational function of degree~$168$ giving the map $j:X_E(7)\to
X(1)$.

The base point $P_E=[0:1:0]\in X_E(7)(K)$ corresponds to the pair
$(E,\id)$.  In \cite{Halberstadt-Kraus-XE7} one can also find explicit formulas
for the rational function $j:X_E(7)\to X(1)$ and for the elliptic curve~$E'$
associated with all 
but finitely many points $P=(x:y:z)\in
X_E(7)$. More precisely, 
explicit polynomials $c_4, c_6 \in \Z[a,b][X,Y,Z]$
of degree~$20$ and~$30$, respectively, are given and the curve $E'$ associated with (all but finitely many)~$P$ has model
\[Y^2=X^3-27c_4(P)X-54c_6(P).\]

The finitely many common zeros of $c_4$
and $c_6$ are the exceptions, which Kraus and Halberstadt treat only
incompletely.  However, in \cite{Fisher} one may find formulas for
four such pairs of polynomials~$(c_4,c_6)$, of which the first is the pair in
\cite{Halberstadt-Kraus-XE7}, and such that at each point $P\in
X_E(7)$ at least one pair $(c_4(P),c_6(P))\not=(0,0)$, thus supplying
us with a model for the associated elliptic 
curve~$E'$ at each point~$P$.

We make use of this model and formulas as follows, given curves $E$,
$E'$ with $E[7]\cong E'[7]$.  Using one curve~$E$ we write down the
model for $X_E(7)$.  Then we find all preimages (if any) of $j'=j(E')$
under the map $X_E(7)\to X(1)$.  While over an algebraically closed
field there are $168$ distinct preimages of each $j'$, except that the
ramification points $j=0$ and $j=1728$ have $56$ and $84$ preimages,
over $\Q$, there are fewer: in the irreducible case there are at
most~$4$ by the results of Section~\ref{S:cong-twist}.

If there are no preimages of~$j'$, we conclude that the isomorphism
$E[7]\cong E'[7]$ is not symplectic.  Otherwise, for each preimage
$P\in X_E(7)(\Q)$ we compute the curve associated to~$P$, which may be
a twist of~$E'$, and test whether it is actually isomorphic to $E'$.
If this holds for one such point~$P$ in the preimage of $j(E')$, then
the isomorphism between $E[7]$ and $E'[7]$ is symplectic.

A similar method may be applied to test for antisymplectic
isomorphisms, using another twist of $X(7)$ denoted $X_E^-(7)$, first
written down explicitly in \cite{PSS}, for which Fisher provides
explicit formulae for the $j$-map and $c_4,c_6$ as above
in~\cite{Fisher}.

We note that it is not necessary to apply both the symplectic and
antisymplectic tests to a triple $(E,E',7)$ if we know already that
$E[7]\cong E'[7]$ as $G_\Q$-modules, since one will succeed if and only if the other
fails (by Proposition~\ref{P:conditionS}).  However we did apply both
tests in our computations with the curves in the database as a test of
our implementation, verifying that precisely one test passes for each
pair.  We also checked that the results obtained for each pair using
the local criteria are the same, so that we can be confident in the
correctness of the results.

These tests have only been carried out using a single curve in each
isogeny class, since we know how to distinguish symplectic from
antisymplectic isogenies.  As a last step, we consider the full isogeny classes
to obtain, for each elliptic curve~$E$ in the database, the complete sets of all curves~$E'$ (non-isogenous to~$E$) which have
symplectically and anti-symplectically isomorphic $7$-torsion modules to~$E$.

The output of this step consists of, for each of the subsets resulting
from Steps~1--3, one or two sets of curves whose union is the set of
all curves in the isogeny classes in the subset.  All curves in the
same set have symplectically isomorphic 7-torsion modules; when there
are two sets, curves in different sets have antisymplectically
isomorphic 7-torsion.

{\bf Example (continued).}  Of the $\numprint{23448}$ non-trivial sets
of isogeny classes with mutually isomorphic irreducible mod~$7$
representations, we find that in $\numprint{16285}$ cases all the
isomorphisms are symplectic, while in the remaining $\numprint{7163}$
cases antisymplectic isomorphisms occur.

Using the local criteria of \cite{FKSym} for $p=11, 13, 17$ we find:
for $p=11$, of the $731$ congruent pairs of isogeny classes, $519$ are
symplectic and $212$ are antisymplectic; for $p=13$, of the $177$
congruent pairs of isogeny classes, $105$ are symplectic and $72$ are
antisymplectic; for $p=17$, all of the $8$ congruent pairs of isogeny
classes are antisymplectic.

\begin{example}
Let $p=7$.  One of the subsets resulting from Steps 1--3 consists of
the pair of isogeny classes $\{\lmfdbec{344025}{bc}{1},
\lmfdbec{344025}{bd}{1}\}$.  Our test shows that
$\lmfdbec{344025}{bc}{1}$ and $\lmfdbec{344025}{bd}{1}$ are
symplectically isomorphic.  The isogeny class
$\lmfdbeciso{344025}{bc}$ contains two $2$-isogenous curves, while
class $\lmfdbeciso{344025}{bd}$ contains only one curve.  Since $2$ is
a quadratic residue mod~7, all three curves have symplectically
isomorphic 7-torsion, and hence Step~4 returns a single set
\[\{\lmfdbec{344025}{bc}{1}, \lmfdbec{344025}{bc}{2}, \lmfdbec{344025}{bd}{1}\}.\]
Another subset resulting from Steps 1--3 is
$\{\lmfdbec{100800}{gw}{1}, \lmfdbec{100800}{hc}{1}\}$.  The same
procedure results in the output of two sets of curves
\[
\{\lmfdbec{100800}{gw}{1}\}, \qquad \{\lmfdbec{100800}{hc}{1},
\lmfdbec{100800}{hc}{2}\},
\]
since our tests show that $\lmfdbec{100800}{gw}{1}$ and
$\lmfdbec{100800}{hc}{1}$ are antisymplectically isomorphic, and the
last two curves are $2$-isogenous.
\end{example}

\begin{example}
Consider the set of six elliptic curves
\[\{\lmfdbec{9225}{a}{1}, \lmfdbec{9225}{e}{1}, \lmfdbec{225}{a}{1},
\lmfdbec{225}{a}{2}, \lmfdbec{11025}{c}{1}, \lmfdbec{11025}{c}{2}\}\]
which form four
complete isogeny classes.  All have isomorphic mod~7 representations
with image the normaliser of a split Cartan subgroup.  The last four
curves all have $j$-invariant~$0$ and CM by $-3$.  The first two are 
$-3$ quadratic twists of each other.

The general methods of Step~4 of this section split this set into two
subsets:
\[
  \{\lmfdbec{9225}{a}{1}, \lmfdbec{225}{a}{1},
  \lmfdbec{11025}{c}{1}\},\qquad \{\lmfdbec{9225}{e}{1},
  \lmfdbec{225}{a}{2}, \lmfdbec{11025}{c}{2}\}
\]
Curves $\lmfdbec{9225}{a}{1}$ and $\lmfdbec{9225}{e}{1}$ give a non-CM
example of Theorem~\ref{T:twist}.  Curves  $\lmfdbec{225}{a}{1}$ and
$\lmfdbec{11025}{c}{1}$ are sextic (but  not quadratic or cubic) twists
and illustrate Theorem~\ref{T:higherTwists}.
\end{example}

\subsection{Auxiliary results for the reducible case}
Compared to the irreducible case, establishing reducible congruences
requires extra work because when working with the semisimplifications
$E[7]^{ss}$ and $E'[7]^{ss}$ important information is lost.  The
objective of this section is to establish Theorem~\ref{T:reducible}
which will allow us to rigorously prove congruences in the reducible
case.

Let $B \subset \GL_2(\Fp)$ be the standard Borel subgroup, i.e. the
upper triangular matrices. Let $H \subset B$ be a subgroup of order
divisible by~$p$.  We can write~$H = D\cdot U$ where $D \subset B$ is
a subgroup of diagonal matrices and $U$ is cyclic generated by
$\mat{1}{1}{0}{1}$.  Moreover, $U$ is a normal subgroup of~$H$ and we
write $\pi : H \to H/U \simeq D$ for the quotient map.

\begin{proposition} \label{P:inner}
Let $H = D \cdot U \subset B$ and $\pi$ be as above. 
Let $\phi$ be an automorphism of~$H$. Assume that  
$\pi(x) = \pi(\phi(x))$ for all $x \in H$. 

Then $\phi$ is given by conjugation in $B$, i.e. there is $A \in B$ such that $\phi(x) = AxA^{-1}$. 
\end{proposition}
\begin{proof}
First note that $\phi$ fixes all scalar matrices in~$H$, since the
assumption on~$\phi$ implies that $\phi(\mat{\lambda}{0}{0}{\lambda})
= \mat{\lambda}{b}{0}{\lambda}$ for some~$b\in\Fp$; however, $b$ must
be~$0$, as otherwise the image has order divisible by~$p$, but $\mat{\lambda}{0}{0}{\lambda}$ has order dividing $p-1$.

Next, let $a\in\Fpstar$ be such that $\mat{a}{0}{0}{1}$ generates~$D$
modulo scalars, which is cyclic; then (using the assumption on~$\phi$
again), $\phi(\mat{a}{0}{0}{1}) = \mat{a}{b}{0}{1}$ for
some~$b\in\Fp$, and $b=0$ if $a=1$.

Finally, we have $\phi(\mat{1}{1}{0}{1}) = \mat{1}{r}{0}{1}$ for
some~$r\in\Fpstar$.

Now set $A=\mat{r}{b(1-a)^{-1}}{0}{1}$ (or $A=\mat{r}{0}{0}{1}$ if
$a=1$ and $b=0$).  A simple check shows that conjugation by~$A$ has
the same effect as~$\phi$ on both $\mat{a}{0}{0}{1}$ and
$\mat{1}{1}{0}{1}$.  Since $H$ is generated by these, together with
scalar matrices, the result follows.
\end{proof}

\begin{proposition} \label{P:fieldF}
Let $p$ be a prime. Let $E/K$ be an elliptic curve such that
$\rhobar_{E,p}$ is reducible. Assume there is an element of order~$p$
in the image of~$\rhobar_{E,p}$.

Then, there is an extension~$F/K$ of degree~$p$, unique up to Galois
conjugacy, such that $E$ acquires a second isogeny over~$F$.
\end{proposition}
\begin{proof} We can choose a basis of $E[p]$ such that
\[
\rhobar_{E,p} =  \begin{pmatrix}
                            \chi& h \\
                            0 & \chi'
                            \end{pmatrix},
\]
where $h\neq0$ since the image contains an element of order~$p$.  One
such element is then $g=\left(\begin{smallmatrix} 1 & 1 \\ 0 &
  1 \end{smallmatrix} \right)$.  Let $H$ be the set of elements
$\sigma \in G_K$ such that $h(\sigma) = 0$. Since $\rhobar_{E,p}$ is a
homomorphism it follows that $H$ is a subgroup of $G_K$, and it has
index~$p$ since the powers of~$g$ are coset representatives.  Let $F
\subset K(E[p])$ be the field fixed by $H$; then $[F : K] =p$.

For the uniqueness, note that precisely one of the $p+1$
one-dimensional subspaces of~$E[p]$ is fixed by~$\rhobar_{E,p}$, and
that~$g$ permutes the remaining subspaces cyclically.  It follows that
$E$ has exactly one $p$-isogeny defined over~$K$, and the remaining
$p$-isogenies are defined over the fixed fields of the conjugate
subgroups~$g^iHg^{-i}$, which are the extensions conjugate to~$F/K$.
\end{proof}

\begin{theorem} \label{T:reducible}
Let $p$ be a prime. Let $E_1, E_2$ be elliptic curves over~$K$ such
that
\begin{itemize}
 \item[(i)] $\rhobar_{E_1,p}^{ss} \simeq \rhobar_{E_2,p}^{ss} \simeq \chi \oplus \chi'$,  where $\chi, \chi' : G_K \to \Fpstar$ are characters;
 \item[(ii)] both $\rhobar_{E_1,p}$ and $\rhobar_{E_2,p}$ have an element of
 order~$p$ in their image.
\end{itemize}
For $i=1,2$, let $F_i/K$ be a degree~$p$ extension where $E_i$
acquires a second isogeny, as given by Proposition~\ref{P:fieldF}.

After replacing $E_2$ by a $p$-isogenous curve if necessary, we have
$\rhobar_{E_1,p} \simeq \rhobar_{E_2,p}$ if and only if $F_1 \simeq
F_2$ (as extensions of~$K$).
\end{theorem}
\begin{proof}
  If $\rhobar_{E_1,p} \simeq \rhobar_{E_2,p}$, then $F_1 \simeq F_2$
  as extensions of~$K$, by the uniqueness part of
  Proposition~\ref{P:fieldF}. We now prove the opposite direction.

From (i) it follows that $\rhobar_{E_i,p}$ is reducible and that,
after replacing $E_2$ by a $p$-isogenous curve if necessary (to swap
$\chi$ with $\chi'$), we have, for $i=1,2$,
\[
\rhobar_{E_i,p} =  \begin{pmatrix}
                            \chi & h_i \\
                            0 & \chi'
                            \end{pmatrix} \quad \text{ with } \quad  h_i : G_K \to \Fp.  
\]
Let $L$ be the field cut out 
by $\chi \oplus \chi'$. 
It follows from (ii) that $h_i|_{G_L} \neq 0$, and
hence the matrix $\left(\begin{smallmatrix}
                            1 & 1 \\
                            0 & 1
                            \end{smallmatrix} \right)$  
is in the image of $\rhobar_{E_i,p}$ for $i=1,2$.
                            
Write $K_i = K(E_i[p])$. 
We have $[K_i : K] = [K_i : L][L : K] = p [L : K]$. 
Since the degree $[L : K]$ divides $(p-1)^2$ it is coprime 
to $p = [F_i : K]$ therefore we have $K_i = L F_i$.

Suppose $F_1 \simeq F_2$. Since $K_1$ is Galois, we have $F_2 \subset K_1$ and therefore $K_p := K_1 = K_2$ is the field cut out by both
$\rhobar_{E_1,p}$ and $\rhobar_{E_2,p}$, i.e. these representations have the same kernel.

Write $G = \Gal(K_p / K)$. From now on we think of $\rhobar_{E_i,p}$ as an injective representation of~$G$. Note that the images of $\rhobar_{E_1,p}$ and $\rhobar_{E_2,p}$
are the same subgroup~$H$ of the Borel.

All the elements in $H$ are of the form $\rhobar_{E_2,p}(\sigma)$ for $\sigma \in G$, so we can consider the map $\phi = \rhobar_{E_1,p} \circ \rhobar_{E_2,p}^{-1} : H \to H$. It is an automorphism of~$H = D\cdot U$ satisfying the hypothesis of Proposition~\ref{P:inner}, where $D$ are 
the matrices $\left(\begin{smallmatrix}
                            \chi & 0 \\
                            0 & \chi'
                            \end{smallmatrix} \right)$.  
Then, $\phi$ is given by conjugation, that is
\[
 \phi(\rhobar_{E_2,p}(\sigma)) = A \rhobar_{E_2,p}(\sigma) A^{-1}.
\]
Since we also have
\[ 
\phi(\rhobar_{E_2,p}(\sigma)) 
=  \rhobar_{E_1,p} \circ \rhobar_{E_2,p}^{-1}(\rhobar_{E_2,p}(\sigma)) = \rhobar_{E_1,p}(\sigma)
\]
we conclude that $\rhobar_{E_1,p}(\sigma) = A \rhobar_{E_2,p}(\sigma) A^{-1}$, as desired.
\end{proof}

\subsection{Distinguishing symplectic from antisymplectic: reducible
  case}

From Steps 1--3 we have (for certain pairs~$(E,E')$) established that
$E[7]^{\sss} \simeq E'[7]^{\sss}$ but this is insufficient to conclude
$E[7] \simeq E'[7]$ when these are reducible $G_\Q$-modules.  To
decide this we will apply Theorem~\ref{T:reducible} and its proof.

Recall that $E[7]$ is reducible if and only if $E$ admits a rational
$7$-isogeny.  Over $\Q$ there is only ever at most one $7$-isogeny,
since otherwise the image of the mod~$7$
representation~$\rhobar_{E,7}$ attached to~$E$ is contained in a split
Cartan subgroup of $\GL(2,\F_7)$, and this cannot occur over~$\Q$
(see~\cite[Theorem~1.1]{GL}).  Furthermore, it is well known that the
size of the $\Q$-isogeny class of~$E$ is either $2$, consisting of two
$7$-isogenous curves, or $4$, consisting of two pairs of $7$-isogenous
curves linked by $2$-{} or $3$-isogenies (but not both). Examples of
these are furnished by the isogeny classes \lmfdbeciso{26}{b},
\lmfdbeciso{49}{a}, and \lmfdbeciso{162}{b} respectively.

Fix an elliptic curve $E$ with $E[7]$ reducible. 
The image
of~$\rhobar_{E,7}$ has the form
\[
  \begin{pmatrix}\chi_1&*\\0&\chi_2  \end{pmatrix},
\]
where~$\chi_1, \; \chi_2 : G_{\Q}\to\F_7^*$ are characters and $*$
(the upper right entry) is non-zero by the previous
discussion. Moreover, the product $\chi_1\chi_2$ is the cyclotomic
character, so in particular $\chi_1\not=\chi_2$. This last observation
is valid over any field not containing $\sqrt{-7}$, so that the
determinant is not always a square.

Now let $E'$ be a second curve such that $E[7]^{\sss}\cong
E'[7]^{\sss}$. 
The image of $\rhobar_{E',7}$ has the form
\[
  \begin{pmatrix}\chi_1'&*'\\0&\chi_2'  \end{pmatrix},
  \]
where $\{\chi_1,\chi_2\}=\{\chi_1',\chi_2'\}$ and $*' \neq 0$ for the
same reason as before.  In particular, there is an element of
order~$7$ in the images of both $\rhobar_{E,7}$ and $\rhobar_{E',7}$.
The next step in applying Theorem~\ref{T:reducible} is to decide if we
need to replace $E'$ with its $7$-isogenous curve to obtain
$\chi_1=\chi_1'$ and $\chi_2=\chi_2'$. For this we determine the
``isogeny characters'' characters $\chi_1$ and $\chi_1'$: the kernel
of $\chi_1$ (respectively $\chi_1'$) cuts out the cyclic extension
of~$\Q$ of degree dividing~$6$ generated by the coordinates of a point
in the kernel of the unique $7$-isogeny from~$E$ (respectively~$E'$).
In this way we can determine whether $\chi_1=\chi_1'$ and
$\chi_2=\chi_2'$ or $\chi_1=\chi_2'$ and $\chi_2=\chi_1'$.  In the
second case, we replace $E'$ with its $7$-isogenous curve, which has
the effect of interchanging $\chi_1'$ and~$\chi_2'$ (as well as
changing $*'$).  Now the image of $\rhobar_{E',7}$ has the form
\[
\begin{pmatrix}\chi_1&*'\\0&\chi_2  \end{pmatrix},
\]
with the same characters, in the same order, as for
$\rhobar_{E,7}$. From Theorem~\ref{T:reducible} we have that
$E[7]\cong E'[7]$ if and only if $F_1 \simeq F_2$, where $F_i$ are the
fields in the statement of Theorem~\ref{T:reducible}.


The field $F_1$ is the common field of definition of all of the other
seven $7$-isogenies from~$E$ (see also Proposition~\ref{P:fieldF}).
The map from $X_0(7)$ to the $j$-line is given by the classical
rational function (see Fricke)
\[
   j = \frac{(t^{2} + 13t + 49) \cdot (t^{2} + 5t + 1)^{3}}{t},
\]
where $t$ is a choice of Hauptmodul for the genus~$0$ curve
$X_0(7)$. Hence the roots of the degree~$8$ polynomial~$(t^{2} + 13t +
49) \cdot (t^{2} + 5t + 1)^{3} -t\cdot j(E)$ determine the fields of
definition of the eight $7$-isogenies from~$E$. In our setting, it has
a single rational root (giving the unique $7$-isogeny from~$E$ defined
over~$\Q$) and an irreducible factor of degree~$7$, which
defines~$F_1$ as an extension of~$\Q$. Similarly, starting from $E'$
we determine $F_2$; finally, we check whether $F_1$ and $F_2$ are
isomorphic.


In this way, for each pair $(E,E')$ whose 
mod~$7$ representations
are reducible with isomorphic semisimplifications, we may determine
whether or not we do in fact have an isomorphism $E[7]\cong E'[7]$,
possibly after replacing $E'$ by its unique $7$-isogenous curve.

In most of the reducible cases encountered in the database, we found
that there was no isomorphism between the $7$-torsion modules
themselves.  In those cases where there is such an isomorphism, we can
determine whether or not it is symplectic using the same methods as in
the irreducible case, noting that the test using the parametrizing
curves $X_E(7)$ and $X_E^-(7)$ do not at any point rely on the
irreducibility or otherwise of the representations.  Finally, if there
are also $2$-{} or $3$-isogenies present we can include these
appropriately, since the former induce symplectic and the latter
antisymplectic congruences.

{\bf Example (continued).} For $p=7$, after Steps 1--3, there
are~$287$ reducible sets of isogeny classes with isomorphic
semisimplification, of size up to~$80$.  Step~5 refines these into
smaller subsets which have actually isomorphic $7$-torsion modules, of
which $384$ are nontrivial.  Among these there are $38$ classes also
admitting a $2$-isogeny and $22$ classes admitting a $3$-isogeny,
making a total of $849$ curves, partitioned into mutually
$7$-congruent subsets of size $2$, $3$ or~$4$: there are $263$ sets of
size~$2$, of which the congruence is symplectic in $142$ cases and
anti-symplectic in $121$ cases; $101$ of size~$3$, of which all the
congruences are symplectic in $56$ cases, and in the remaining $45$
cases, only one congruence is symplectic; and~$20$ of size~$4$, in
which all congruences are symplectic in $8$ cases, there are two pairs
of symplectically congruent curves (with congruences between curves in
different pairs being anti-symplectic) in $2$ cases, and in $10$ cases
there are $3$ curves mutually symplectically congruent and
anti-symplectically congruent to the fourth curve.

For $p\ge11$ there are no reducible cases to consider.

\subsection{Twists}
If there is a mod~$p$ congruence between two elliptic curves~$E_1$
and~$E_2$, then for any $d\in\Q^*$ there will also be a congruence
(with the same symplectic type) between their quadratic twists
$E_1^d$ and~$E_2^d$ (see \cite[Lemma~11]{FKSym}).
Nevertheless, it is hard to say precisely how many congruences there
in the database ``up to twist'', since twisting changes conductor (in
general), so we may have a set of mutually $7$-congruent elliptic
curves in the database, but with one or more of their twists not in
the database, so the twisted set in our data will be smaller.

Instead, to have a measure of how many congruences we have found up to
twist, we simply report on how many distinct $j$-invariants we found,
excluding as before curves which are only congruent to isogenous
curves.  For $p=7$ there are $\numprint{11761}$ distinct
$j$-invariants of curves with irreducible mod~$7$ representations
which are congruent to at least one non-isogenous curve, and $154$
distinct $j$-invariants in the reducible case.

For $p=11$ there are $212$ distinct $j$-invariants and for $p=13$
there are $39$.  For $p=17$, all $17$-congruent isogeny classes
consist of single curves, the eight pairs are quadratic twists, and
the $j$-invariants of the curves in each pair are
$48412981936758748562855/77853743274432041397$ and $-46585/243$.  One
such pair of $17$-congruent curves consists of $\lmfdbec{47775}{b}{1}$
and~$\lmfdbec{3675}{b}{1}$.

\section{Evidence for the Frey-Mazur conjecture}
\label{S:Frey-Mazur}

Theorem~\ref{T:cong19} states that the strong form of the
Frey--Mazur conjecture with~$C=17$ holds for the congruences available
in the LMFDB database. We refer to~\cite{Halberstadt-Kraus-FreyMazur}
for one theoretical result towards this very challenging and still
open conjecture.

\begin{proof}[Proof of Theorem~\ref{T:cong19}]
We must prove that if $p\ge19$ then the only $p$-congruences between
elliptic curves of conductor at most $\numprint{500000}$ are those
induced by isogenies.

Let~$p \geq 5$ be a prime. Let $N_E$ and $\Delta_E$ denote the
conductor and the minimal discriminant of $E$, respectively. Write
also $\tilde{N}_E$ to denote~$N_E$ away from~$p$ and let $N_p$ be the
Serre level (i.e. the Artin conductor away from~$p$)
of~$\rhobar_{E,p}$.  We have $N_p \mid \tilde{N}_E$.

Recall that the conductor of an elliptic curve at primes~$p \geq 5$ divides~$p^2$. Moreover, from Kraus~\cite[p. 30]{KrausThesis}
it follows that, for each $\ell \neq p$, if 
$\vv_{\ell}(N_p) \neq \vv_{\ell}(N_E) = \vv_{\ell}(\tilde{N}_E)$ then
$\vv_{\ell}(N_E) = 1$ and $p \mid \vv_{\ell}(\Delta_E)$.
Therefore, we can find primes $q_i \nmid pN_p$ such that 
\begin{equation}\label{E:condE}
  N_E = p^s \cdot N_p \cdot q_0 \cdot \ldots \cdot q_n, 
 \qquad p \mid \vv_{q_i}(\Delta_E), \qquad 0 \leq s \leq 2
\end{equation}
where the number of~$q_i$ occurring is $\geq 1$ if and only if 
$\tilde{N}_E \neq N_p$. 

Now let $E'/\Q$ be another elliptic curve satisfying $E[p] \simeq E'[p]$ as $G_\Q$-modules. Write $N_{E'}$, $\tilde{N}_{E'}$, $\Delta_{E'}$, $N'_p$ and $\rhobar_{E',p}$ to denote analogous quantities attached to~$E'$. We have $N'_p \mid \tilde{N}_{E'}$.

By assumption, we have $\rhobar_{E',p} \simeq \rhobar_{E,p}$ so these representations have the same Serre level, i.e. $N_p = N_p'$ and (similarly as for~$E$) we can find primes~$q_i' \nmid pN'_p$ such that~$N_{E'}$ factors as
\begin{equation}\label{E:condE'}
 N_{E'} = p^{s'} \cdot N_p \cdot q'_0 \cdot \ldots \cdot q'_m, 
 \qquad p \mid \vv_{q'_i}(\Delta_E'),
 \qquad 0 \leq s' \leq 2.
\end{equation}
The representations $\rhobar_{E,p}$ and $\rhobar_{E',p}$ also have the same Serre weights $k$ and $k'$, respectively. Note that for $s=0$ ($E$ has good reduction at~$p$) we have $k=2$ and for $s=1$ ($E$ has multiplicative reduction at~$p$) we have $k=2$ if $p \mid \vv_p(\Delta_E)$ or $k=p+1$ otherwise (see for example \cite[p. 3]{KrausThesis}); moreover, for $s=2$ it follows from \cite[Th\'eor\`eme 1]{KrausThesis} that $k \not\in \{2, p+1\}$ 
for $p \geq 19$. Similar conclusions apply to $E'$, $s'$ and $k'$. 
Therefore, we have 2 cases: (i) if $s = 2$ or $s=1$ and $p \nmid \vv_{p}(\Delta_E)$ then $s'=s$; (ii) if $s=0$ or $s=1$ and $p \mid \vv_{p}(\Delta_E)$ then $s' \in \{0,1 \}$.

Suppose $E'$ is a non-isogenous curve with the same conductor. Taking differences of traces of Frobenius at different primes shows that 
there are no congruence between any two of them for $p \geq 19$, otherwise $p$ needs to divide the differences (see (1) below). 
Thus $N_E \neq N_{E'}$.

Suppose $\tilde{N}_E = \tilde{N}_{E'}$, so that the only difference in the conductors is at~$p$. From the possibilities above for the Serre weights, after interchanging $E$ and $E'$ if needed, 
we can assume $s=1$ and $s'=0$ and we also know that $p \mid \vv_p(\Delta_E)$. On the other hand, if $\tilde{N}_E \neq \tilde{N}_{E'}$ then, after interchanging $E$ and $E'$ if needed, we have $N_p \neq \tilde{N}_E$ and so there is at least one prime~$q_i \neq p$ appearing in the factorization~\eqref{E:condE}, which in particular satisfies $p \mid \vv_{q_i}(\Delta_E)$.

Let $\calM_E$ be the set of pairs $(q,p)$ where $q$ is a multiplicative prime of~$E$ and
$p \geq 19$ is a prime satisfying $p \mid \vv_{q}(\Delta_E)$.
Note that we can have $q=p$. Let $\calM_{E'}$ be the analogous set for~$E'$. From the previous paragraph we conclude that 
$p$ has to occur in the second entry of one of the pairs~$(q,p)$ in~$\calM_E$ or $\calM_{E'}$.

To complete the proof, we carried out the following computations on
the LMFDB database of all elliptic curves defined over~$\Q$ and
conductor at most~$\numprint{500000}$:
\begin{enumerate}
\item For each $N\le \numprint{500000}$ and each pair of non-isogenous
  curves~$E_1,E_2$ of conductor~$N$ (if there are at least two such
  isogeny classes), we computed $\gcd_{\ell\le B,
    \ell\nmid N}(a_{\ell}(E_1)-a_{\ell}(E_2))$ for increasing~$B$
  until the value of the $\gcd$ was~${}\le17$.  The success of this
  computation shows that there are no congruences mod~$p$ between
  non-isogenous curves of the same conductor for $p\ge19$.
\item For one curve~$E$ in each isogeny class we computed the set
  $\calM_E$ from the conductor and minimal discriminant.  We found that
  the largest prime~$p$ occurring in any~$\calM_E$ was~$97$: in fact,
  all~$p$ with $19\le p\le97$ occur except for~$p=89$.  Hence any
  mod~$p$ congruence between non-isogenous curves in the database must
  have~$p\le97$.  In view of the computations of
  Section~\ref{S:statistics}, there are no such congruences for $19\le
  p\le97$.

  Note that the set~$\calM_E$ is unchanged if we replace $E$ by a curve
  isogenous to it, provided that the isogeny has degree divisible only
  by primes less than~$19$.  But the only curves defined over~$\Q$ with
  isogenies of prime degree~$p\ge19$ are the CM curves for
  $p=19,43,67,163$, which have no multiplicative primes, and the pairs
  of $37$-isogenous curves, which have the same property (the smallest
  conductor being $1225=5^2\cdot7^2$).  Hence in this step it suffices
  to consider just one curve in each isogeny class.
\end{enumerate}
\end{proof}



\end{document}